\makeatletter

\def\eqalign#1{\,\vcenter{\openup\jot\m@th
  \ialign{\strut\hfil$\displaystyle{##}$&$\displaystyle{{}##}$\hfil
      \crcr#1\crcr}}\,}
\def\eqalignno#1{\displ@y \tabskip\@centering
  \halign to\displaywidth{\hfil$\displaystyle{##}$\tabskip\z@skip
    &$\displaystyle{{}##}$\hfil\tabskip\@centering
    &\llap{$##$}\tabskip\z@skip\crcr
    #1\crcr}}
\def\leqalignno#1{\displ@y \tabskip\@centering
  \halign to\displaywidth{\hfil$\displaystyle{##}$\tabskip\z@skip
    &$\displaystyle{{}##}$\hfil\tabskip\@centering
    &\kern-\displaywidth\rlap{$##$}\tabskip\displaywidth\crcr
    #1\crcr}}

\makeatother

\newdimen\pixel \pixel=.00333333 true in
\documentclass[11pt]{article}
\usepackage{float}
\usepackage{amsmath}
\usepackage{amssymb}
\usepackage{epsfig}
\usepackage{fullpage}
\usepackage{color}
\usepackage{tikz}

\definecolor{light-gray}{gray}{0.3}

\makeatletter
\def\bigpar{\bigbreak\@afterindentfalse\@afterheading\ignorespaces}
\def\medpar{\medbreak\@afterindentfalse\@afterheading\ignorespaces}
\def\smallpar{\smallbreak\@afterindentfalse\@afterheading\ignorespaces}

\newlength{\saveindent}
\newenvironment{proof}%
      {\bigpar{\bf Proof:}\ %  previously \sentsp rather than \
             \setlength{\saveindent}{\parindent} %\setlength{\parindent}{0pt}%
                       \ignorespaces}%
      {\stopproof\ignorespaces\bigbreak \setlength{\parindent}{\saveindent}}

      {\bigpar{\bf Proof:}%
             \setlength{\saveindent}{\parindent} %\setlength{\parindent}{0pt}%
                       \,\ignorespaces}%
      {\ignorespaces\bigbreak \setlength{\parindent}{\saveindent}}

      {\bigpar{\bf Proof:}\ %
             \setlength{\saveindent}{\parindent} %\setlength{\parindent}{0pt}%
                       \ignorespaces}%
      {\ignorespaces\bigbreak \setlength{\parindent}{\saveindent}}

\newenvironment{proofof}[1]%
      {\bigpar{\bf#1:}\ %
             \setlength{\saveindent}{\parindent} %\setlength{\parindent}{0pt}%
                       \ignorespaces}%
      {\stopproof\ignorespaces\bigbreak \setlength{\parindent}{\saveindent}}

\newenvironment{remark}%
      {\smallpar{\bf Remark:}\ %  **** \bigpar, and \bigbreak below ****
                       \ignorespaces}%
      {\stopproof\ignorespaces\medbreak \setlength{\parindent}{\saveindent}}

\newenvironment{remark*}%
      {\smallpar{\bf Remark:}\ %  **** \bigpar, and \bigbreak below ****
                       \ignorespaces}%
      {\ignorespaces\medbreak \setlength{\parindent}{\saveindent}}

      {\smallpar{\bf Remarks:}\ % **** ditto ****
                       \ignorespaces}%
      {\stopproof\ignorespaces\medbreak \setlength{\parindent}{\saveindent}}

\newenvironment{remarks*}%
      {\smallpar{\bf Remarks:}\ % **** ditto ****
                       \ignorespaces}%
      {\ignorespaces\medbreak \setlength{\parindent}{\saveindent}}

      {\smallpar{\bf Remark:}\ %
                       \ignorespaces}%
      {\stopproof\ignorespaces\medbreak \setlength{\parindent}{\saveindent}}

      {\smallpar{\bf Remarks:}\ %
                       \ignorespaces}%
      {\stopproof\ignorespaces\medbreak \setlength{\parindent}{\saveindent}}
\makeatother

\newtheorem{alemma}{Lemma}

\newtheorem{theorem}{Theorem}%[section]    % Remove percent sign
\newtheorem{lemma}[theorem]{Lemma}

\newtheorem{proposition}[theorem]{Proposition}
\newtheorem{corollary}[theorem]{Corollary}

\newtheorem{example}{Example}
\def\begex{\begin{example}\parindent=0pt \rm}
\def\endex{\end{example}}
\def\square{\vbox{\hrule height.2pt\hbox{\vrule width.2pt height5pt \kern5pt
                                   \vrule width.2pt} \hrule height.2pt}}
\def\stopproof{\hfill \square \smallskip}

\def \zplus {{z_R}}
\def \zminus {{z_L}}

\def \pih {{\widehat \pi}}

\def \th {{\widehat T}}
\def \ttil {{\widetilde T}}

\def \config {{\cal C}}

\def \fin {{{\rm final}}}

\def \lhat {{ \widehat L}}

\def \aux {{primary }}
\def \parity {{\rm parity}}

\def \tmt {{t \wedge T}}

\def \gt {{ \widetilde G}}
\def \vt {{ \widetilde V}}

\def\fif {{fifteen puzzle}}

\def \Pone {{ P_{x_1}}}
\def \fone {{ f_{x_1}}}

\def\sfrac#1#2{{\textstyle{#1 \over #2}}}

\def\bigo{{\rm O}}
\def\littleo{{\rm o}}

\def \tmix {T_{\rm mix}}

\def\half{{\textstyle{1\over2}}}

\def\quarter{{\textstyle{1\over4}}}

\def \ff {{f_{\rm final}}}
\def \fhat {{\widehat f}}

\def \fhf {{\fhat_{\rm final}}}

\def \wdist {{W_{\rm dist}}}

\def \hhat {{\widehat H}}
\def \yhat {{\widehat Y}}
\def \ahat {{\widehat R_s}}

\def \bhat {{\widehat L_s}}

\def \r {{\bf R}}
\def \var {{ \rm var }}

\def \r {{ \cal F}}
\def \d {\partial}
\def\r|{{\Bigr\vert}}
\def\l|{{\Bigl\vert}}

\def \R {{\bf R}}

\def\phi {\Phi}

\def\e{\epsilon}

\def \cov {{\rm cov}}
\def \sd {{\rm sd}}

\def \ttil {{\mathcal S}}
\def \pt {{\tilde p}}
\def \at {{\tilde \alpha}}
\def \Pt {{\tilde P}}
\def \pit {{\widetilde \pi}}
\def \ft {{\tilde f}}
\def \alphat {{\tilde \alpha}}

\def\config {\cal C}

\def\tmix{\tau_{\rm mix}}

\def\varepsilon{\mathchar"122 }

\def \one {{\mathbf 1}}
\def \chi {{\mathbf 1}}

\def\Gamma {{}}

\def\p {{ \mathbb P}}
\def\P {{ \mathbb P}}
\def\ph {{{\widehat V_n}}}

\def\e {{ \mathbb {E}}}

\def \phit {{\widetilde \phi}}

\def \tb {{ T_B}}

\def \tc {{ T_C}}

\def\Sq{{\cal S}_q}
\def\Sq-{{\cal S}_{q-1}}

\def \et { {\tilde  {\cal E}}}

\def \taut {{ \tau}}

\def \nt {{N}}

\def \gstar {{G^*}}

\def\xhat{\widehat{X}}

\def \z {{\bf Z}}
\def \Z {{\bf Z}}

\def \sm {{\setminus}}

\def\pim{\pi_*}

\def\f {{\cal F}}
\def\h {{\cal H}}
\def\hh {{\widehat {\cal H}}}

\def\given {{\,|\,}}
\def\Given {{\,\Bigl|\,}}

\def\one{{\mathbf 1}}

\def\ent{{ \rm \sc ENT}}
\def\entpi {{\ent_{\pi}}}
\def\entpit {{\ent_{\pit}}}

\def \zhat {{ \widehat Z}}

\def \ore {{$\Omega$-restricted }}
\def \bgb {{BGB }}
\def \ct {{ \widetilde C }}

\newcommand{\lab}{\label}

\newcommand{\be}{\begin{eqnarray}}
\newcommand{\ee}{\end{eqnarray}}
\newcommand{\eps}{{\cal E}}

\def \eone {{ \eps_{x_1}}}

\begin{document}
\title{The mixing time of the \fif}
\author{
{\sc Ben Morris}\thanks{Department of Mathematics,
University of California, Davis.
Email:
{\tt morris@math.ucdavis.edu}.
Research partially supported by
NSF grant DMS-1007739.}
\and
{\sc Anastasia Raymer}\thanks{Department of Mathematics,
Cornell University.
Email:
{\tt araymer@math.cornell.edu}.
}  }
\date{}
\maketitle
\begin{abstract}
\noindent
We show that there are 
universal positive constants $c$ and $C$ such 
the mixing time $\tmix$ for the fifteen puzzle in an 
$n \times n$ torus satisfies $c n^4 \log n \leq \tmix \leq C n^4 \log^2 n$.
\end{abstract}

\setcounter{page}{1}
%-------------------------------------------------------------------------

%-------------------------------------------------------------------------

\section{Introduction} 
The fifteen puzzle, often credited to Sam Loyd, 
was a craze in 1880. The game consists of a 
$4 \times 4$ grid with fifteen tiles, labeled 1,2,$\dots$, 15, 
and an empty space (the ``hole''). In a move, the player pushes
a tile into the hole. The tiles start in ``mixed up'' order and the 
goal is to sort the tiles and move the hole to the lower right 
corner, as shown in Figure \ref{figa}. 
There are also $3 \times 3$ 
and $2 \times 4$ versions of the game. 
\begin{figure}[t]
\centering
\begin{tikzpicture} 
  \foreach \y in {0,2}
  \draw [fill = white] (\y, 0) rectangle (\y + 1, 1);
  \foreach \y in {1,3}
  \draw [fill = white] (\y, 1) rectangle (\y + 1, 2);
  \foreach \y in {1}
  \draw [fill = white] (\y, 0) rectangle (\y + 1, 1);
  \foreach \y in {3}
  \draw [fill = white] (\y, 0) rectangle (\y + 1, 1);
  \foreach \y in {1}
  \draw [fill = white] (\y, 2) rectangle (\y + 1, 3);
  \foreach \y in {3}
  \draw [fill = gray] (\y, 2) rectangle (\y + 1, 3);
  \foreach \y in {1,3}
  \draw [fill = white] (\y, 3) rectangle (\y + 1, 4);
  \foreach \y in {0,2}
  \draw [fill = white] (\y, 1) rectangle (\y + 1, 2);
  \foreach \y in {0,2}
  \draw [fill = white] (\y, 2) rectangle (\y + 1, 3);
  \foreach \y in {0,2}
  \draw [fill = white] (\y, 3) rectangle (\y + 1, 4);

\node at (.5,1.5) {14};
\node at (1.5,1.5) {2};
\node at (2.5,1.5) {7};
\node at (3.5,1.5) {6};

\node at (.5,.5) {15};
\node at (1.5,.5) {13};
\node at (2.5,.5) {12};
\node at (3.5,.5) {5};

\node at (.5,2.5) {11};
\node at (1.5,2.5) {10};
\node at (2.5,2.5) {1};
%\node at (3.5,2.5) {7};

\node at (.5,3.5) {4};
\node at (1.5,3.5) {9};
\node at (2.5,3.5) {8};
\node at (3.5,3.5) {3};

\node at (2, -.5) {Start};

\def\o{8}

  \foreach \y in {0,2}
  \draw [fill = white] (\y + \o, 0) rectangle (\y + \o + 1, 1);
  \foreach \y in {1,3}
  \draw [fill = white] (\y + \o, 1) rectangle (\y + \o + 1, 2);
  \foreach \y in {1}
  \draw [fill = white] (\y + \o, 0) rectangle (\y + \o + 1, 1);
  \foreach \y in {3}
  \draw [fill = gray] (\y + \o, 0) rectangle (\y + \o + 1, 1);
  \foreach \y in {1}
  \draw [fill = white] (\y + \o, 2) rectangle (\y + \o + 1, 3);
  \foreach \y in {3}
  \draw [fill = white] (\y + \o, 2) rectangle (\y + \o + 1, 3);
  \foreach \y in {1,3}
  \draw [fill = white] (\y + \o, 3) rectangle (\y + \o + 1, 4);
  \foreach \y in {0,2}
  \draw [fill = white] (\y + \o, 1) rectangle (\y + \o + 1, 2);
  \foreach \y in {0,2}
  \draw [fill = white] (\y + \o, 2) rectangle (\y + \o + 1, 3);
  \foreach \y in {0,2}
  \draw [fill = white] (\y + \o, 3) rectangle (\y + \o + 1, 4);

\node at (\o +.5,1.5) {9};
\node at (\o +1.5,1.5) {10};
\node at (\o +2.5,1.5) {11};
\node at (\o +3.5,1.5) {12};

\node at (\o +.5,.5) {13};
\node at (\o +1.5,.5) {14};
\node at (\o +2.5,.5) {15};
%\node at (\o +3.5,.5) {5};

\node at (\o +.5,2.5) {5};
\node at (\o +1.5,2.5) {6};
\node at (\o +2.5,2.5) {7};
\node at (\o +3.5,2.5) {8};

\node at (\o +.5,3.5) {1};
\node at (\o +1.5,3.5) {2};
\node at (\o +2.5,3.5) {3};
\node at (\o +3.5,3.5) {4};

\node at (\o +2, -.5) {End};
\end{tikzpicture}
\caption{}
\label{figa}
\end{figure}
In this paper we study the problem,
posed by Diaconis \cite{diaconis}, of finding the  {\it mixing time}
of the fifteen puzzle: starting from a solved 
game, how many steps are required to ``mix up'' the tiles again,
if at each step we choose a move uniformly at random?
(See Section \ref{ls} for a precise definition of the mixing time). 
%More, precisely, if we play the game in an $n \times n$ grid,
%what is the dependence of the mixing time on the parameter $n$?
%This problem was posed by Diaconis in \cite{diaconis}.

We can define the fifteen puzzle on any finite graph 
$G$ as follows. In a configuration, 
the tiles and hole occupy the vertices of $G$. In a move, the hole 
is interchanged with a tile in an adjacent vertex. If $G$ 
is bipartite, then there are some configurations that are not reachable
from a given starting state. To see this, suppose that $G$ is bipartite, so 
that we can define a parity for each vertex in $G$. 
If we view configurations as permutations $\pi$ on the vertex set of $G$,
and define
\be
\label{odef}
\Omega = \{\pi: \parity(\pi) = \parity(\mbox{hole})\};
\ee
\be
\label{ocdef}
\Omega^c = \{\pi: \parity(\pi) \neq \parity(\mbox{hole})\};
\ee
then it is impossible to transition between $\Omega$ and $\Omega^c$, using 
a legal move. Suppose that the game is started in a configuration 
in $\Omega$. We say the game is {\it solvable} if every configuration 
in $\Omega$ is reachable by legal moves. 
If $G$ is not bipartite, 
we say the game is solvable if every configuration 
is reachable by legal moves. 
The fifteen puzzle is known to be solvable on most graphs (see \cite{wil}); 
in particular, it is solvable on an $m \times n$ grid 
provided that $m$ and $n$ are both at least $2$ (see \cite{js}).

In the present paper, we analyze the fifteen puzzle in the 
$n \times n$ torus $G_n := \z_n^2$. We consider the Markov
chain, which we call the {\it Loyd process,}
defined by the following transition rule:
\begin{enumerate}
\item with probability $\half$, do nothing; else
\item choose 
a uniform random move
and make it. 
\end{enumerate} 
(We have added a holding probability of $\half$ to avoid periodicity.)
The Loyd process is related to the
{\it interchange
process} on $G_n$, which is defined as follows. 
In a configuration, each vertex in $G_n$ 
is occupied by a particle. At each step, choose 
a pair of neighboring particles 
uniformly at random and then 
interchange them.
Yau \cite{yau} famously showed that the log Sobolev constant
(see Section \ref{ls} for 
a precise definition) 
for 
the interchange process is on the order of $n^{-4}$, which implies
that the mixing time is $\bigo(n^4 \log n)$, and there is a matching lower 
bound \cite{ex}.
The Loyd process can be viewed as a variant of the 
interchange process, where there is a special particle
(the hole) that is conditioned to be involved in each step.

Our main result is to determine the mixing time 
of the Loyd process to within
a factor of $\log n$. We show that there are universal constants
$c > 0$ and $C>0$ such that the mixing time $\tmix$ 
for the Loyd process in $G_n$ satisfies
\[
c n^4 \log n \leq \tmix \leq C n^4 \log^2 n.
\]
For the upper bound, we use the comparison techniques for random
walks on groups developed in \cite{ct}, which allow us to bound the 
log Sobolev constant for the Loyd process using known bounds for 
shuffling by random transpositions. 
A 
difficulty that arises here is that $G_n$ is bipartite when 
$n$ is even, which implies that there is a restricted state space. 
To handle this, we develop a method to compare log Sobolev constants across 
different state spaces. To compare our chain with shuffling by 
random transpositions, we introduce three intermediate chains
and then make a total of four comparisons. 

For the lower bound, we use a variation on Wilson's method
\cite{wilson}. Wilson's method is useful 
%for proving lower bounds
%on the mixing time
when the Markov chain can be described 
as a system with a large number of particles where 
the motion of each individual particle is itself a Markov chain.
(In the Loyd process the movement of a single tile 
is {\it not} a Markov chain; however, we can get around this by considering 
the process only at times when the hole is to its immediate right.) 
In Wilson's method, one often analyzes a distinguishing statistic
of the form 
\[
\sum_p f(\mbox{position of particle $p$}),
\] 
where the sum is over  a 
certain set of particles, and $f$ is an eigenfunction for the motion 
of a single particle. In a typical application of Wilson's method,
only a bounded number of particles are involved in each move, and
hence the distinguishing statistic is slowly decaying. 
However, in the Loyd process, each move of the hole affects the 
distribution of the final position of each tile, which 
makes the ``Wilson statistic'' hard to analyze. Fortunately,
by making use of some surprising cancellations we are able 
to prove a lower bound of the expected form $c n^4 \log n$. 

\section{Mixing time, log Sobolev constant and the 
harmonic extension}
\label{ls}
Let 
$(X_0, X_1, \dots)$ be an 
irreducible, aperiodic Markov chain 
on a finite state space $S$
with transition probabilities 
$p(x,y)$, and suppose that the stationary distribution
$\pi$ is uniform over $S$.
For probability measures $\mu$ and $\nu$ on $S$,
define the {\it total variation distance}
$|| \mu - \nu || = \half \sum_{x \in S} |\mu(x) - \nu(x) |$, 
and 
define the {\it $\epsilon$-mixing time}
\be
\label{mixingtime}
\tmix(\epsilon) = \min \{t: || p^t(x, \, \cdot) - \pi || \leq \epsilon 
\mbox{ for all $x \in S$}\} \,.
\ee
\noindent
The  {\it mixing time}
is $\tmix = \tmix(e^{-1})$.

For $f: S \to \R$ define 
\[
\e_\pi(f) = \sum_{x \in S} f(x) \pi(x),
\]
and
\[
\ent_\pi(f) = \e_\pi \left( f \log{f \over \e_\pi(f)} \right),
\]
and define 
the Dirichlet form 
\[
\eps( f, f) = \half \sum_{x,y \in S}  \pi(x) p(x,y) 
\left( f(x) - f(y) \right)^2.
\]
The log-Sobolev constant is defined 
by 
\[
\alpha = \min_{f: \ent_\pi(f^2) \neq 0} { \eps(f, f) \over \ent_\pi( f^2) }.
\]
The mixing time is related to the log Sobolev constant 
via the following inequality \cite{ls}:
\begin{equation}
\label{logsob}
\tmix \leq {4 + \log \log |S| \over 4 \alpha}.
\end{equation}

For $S' \subset S$,  
let $\tau_1 < \tau_2 < \cdots$ be the times when the chain is in $S'$.
The {\it restriction} of the Markov chain to $S'$ 
is the new Markov chain
$(X_{\tau_1}, X_{\tau_2}, \dots)$. 
For $f: S' \to \R$, the {\it harmonic extension 
of $f$ to $S$} is the function $\ft$ that agrees with $f$ on $S'$ 
and is harmonic on $S \setminus S'$, which can be defined by
\[
\ft(x) 
 = \left\{\begin{array}{ll}      
  f(x)          & \mbox{if $x \in S'$;} \\
\e_x(f(X_{T_{S'}}))           & \mbox{otherwise,} \\
\end{array}
\right.
\]
where $\e_x\Bigl(\, \cdot \, \Bigr) := 
\e\Bigl(\, \cdot \, \Given X_0 = x \Bigr)$ and  
$T_{S'} = \min \{t \geq 0: X_t \in S'\}$ is the hitting time of $S'$.

\section{Random walks on groups and comparison techniques}
%Random walks on groups are usually defined in the following way.
Let $G$ be a finite group and let $p$ be a probability measure 
supported on a set of generators of $G$. The random walk on $G$ 
driven by $p$ is the Markov chain with the following transition rule. 
If the current state is $x$, choose $y$ at random
according to $p$, and then
move to $xy$.

In the present paper we shall use a slightly 
%different, and 
more general definition of a random walk on a group. 
For a finite group $G$, we write $\gstar$ for the set of 
{\it strings} over $G$, that is, finite sequences of 
elements of $G$. 
If 
$g_1 g_2 \cdots g_k \in \gstar$, we define its
{\it evaluation} as the 
group element $g_1 \cdot g_2 \cdots g_k$
(where $\cdot$ is the group operation). 
As an abuse of notation, we use  
the string itself as notation for its evaluation.
%When there is little danger of confusion, we use 
%a string itself as notation for its evaluation.  
(Thus there exist strings $y$ and $y'$ 
such that 
$y \neq y' \in \gstar$, but $y = y'$ in $G$.) 
If two 
strings evaluate to the same group element, we say that one is a 
{\it representation} of the other. 

Let $H$ be a subgroup of $G$, let $p$ be a probability 
measure on $G^*$, and 
suppose that 
$$\{ g \in G: \mbox{$g$ is the evaluation of
a string in the support of $p$}\}$$
is a generating 
set for $H$. 
The random walk on $H$ driven by $p$ is the Markov chain with
the following transition rule. If the current state is $x \in H$: 
\begin{enumerate}
\item 
choose the string $y$ at random according to $p$;
\item move to $xy$. 
\end{enumerate}
%Note that this Markov chain is a random walk on a group in the 
%usual sense; however, there is additional information for each step,
%which is the 
%the string $g_1 \cdots g_k$. 
We call strings 
in the support of $p$ {\it moves}. 

%If $p$ is a probability measure on
%$\gstar$, 
%For a probability measure $p$ on $\gstar$, 
The Dirichlet
form for the random walk on $H$ driven by $p$ can be written
\[
\eps_p( f, f) = {1 \over 2 |H|} \sum_{x \in H, y \in \gstar} 
\left( f(x) - f(xy) \right)^2 p(y).
\]
%Here the $y$ in $f(xy)$ refers to the evaluation of $y$. 
For $x$ and $y$ in $\gstar$ we write $xy$ for the concatenation 
of $x$ and $y$. 

\subsection{Comparison techniques}
We say that $p$ is {\it symmetric} 
if $
p( g_1\cdots g_k) = 
p( g^{-1}_k\cdots g^{-1}_1)$
for every $g_1 \cdots g_k \in \gstar$.
Let $p$ and $\pt$ be symmetric probability measures 
on $\gstar$
that drive random walks on a subgroup $H$ of $G$.
Think of $\pt$ as driving a known chain 
and $p$ as driving an unknown chain. 
Let $E$ be the support of $p$. 
For each $y$ in the support of $\pt$, we give a {\it random} representation 
of $y$ of the form $Z_1 Z_2 \cdots Z_K$, where $K$ is possibly
random, and each of the $Z_i$ are random elements of $E$. 
Given such a representation, we write $|y|$ for the value of $K$. 
For $z \in E$, let
\begin{eqnarray*}
N(z,y) &=& \mbox{number of times $z \in E$ occurs} \\
       & & \mbox{in the representation of $y$.}
\end{eqnarray*}
\begin{theorem}(\cite{dsrep})
\label{gt}
The Dirichlet forms for the random 
walks driven by $\pt$ and $p$, respectively, satisfy
\[
\et \leq A \eps
\]
with 
\[
A = \max_{z \in E} {1 \over p(z)} \e \Bigl( \sum_{y \in \gstar} 
|y| N(z,y) \pt(y) \Bigr) \,.
\]
\end{theorem}
\begin{remark}
Note that the quantity $A$ can be written as
\[
A = \max_{z} {1 \over p(z)} \e \left( N(Y, z) |Y| \right),
\]
where $Y$ is chosen at random according to $\pt$. 
\end{remark}
Since the denominator in the definition of log Sobolev constant 
is the same whether the random walk is driven by $p$ or $\pt$, 
Theorem \ref{gt} yields:

\begin{corollary}
\label{lscomp}
Let $A$ be as in Theorem \ref{gt}. The log Sobolev 
constants for the walks driven by $\pt$ and $p$, respectively,
satisfy
\[
\at \leq A \alpha.
\]
\end{corollary}

\section{Mixing time upper bound: main theorem}
Before stating the mixing time upper bound,
we give a more formal description of the
Loyd chain, and we also describe 
some other chains that are used in comparisons. 
Suppose $n \geq 2$ and let $V_n$ 
be the vertex set of the $n \times n$ torus 
$G_n$. 
%To ensure solvability of the fifteen puzzle on $G_n$,
%we assume that $n \geq 2$.  
Note that if we give each tile and the hole a unique label in $V_n$, 
then we can view configurations as permutations on $V_n$. 
For reasons that will become clear later, we give the hole the label 
$h := (0,0)$. 
%Let $h = (0,0)$ be the label of the hole. 
For 
$y = (y_1, y_2) \in V_n$, call $y$ {\it even} if $y_1 + y_2$ is even, and 
define $\Omega$ and $\Omega^c$ as in equations (\ref{odef}) and 
(\ref{ocdef}). 
Since the fifteen puzzle is solvable in 
a grid of size $2 \times 2$ or larger, any pair of states
in $\Omega$ (respectively, $\Omega^c$) communicate. Furthermore, there are 
transitions between $\Omega$ and $\Omega^c$ if and only if $n$ is odd.
It follows that the state space is restricted to 
half the permutations exactly when $n$ is even. 
If we start from a configuration in $\Omega$, then 
the state space is
\[ 
 \left\{\begin{array}{ll}      
\Omega             & \mbox{if $n$ is even;} \\
\mbox{all permutations on $V_n$}  
& \mbox{if $n$ is odd.} \\
\end{array}
\right.
\]

As stated in the Introduction, we prove the upper bound
by comparing the Loyd chain with shuffling by random transpositions,
using a number of intermediate chains. For easy reference
we give a short description of each of these chains below.
For each of these chains there is an implicit holding probability 
of $\half$. That is, at each step we do nothing with probability 
$\half$; 
else make the move described.
\begin{enumerate}
\item Loyd chain: interchange the hole with one of four adjacent tiles, 
chosen uniformly at random. 

\item Hole-conditioned chain (HC): 
Interchange the hole with a tile chosen uniformly at random.

\item Shuffling by random transpositions (RT): 
Choose two particles uniformly at random and then swap them.
(Here {\it particle} refers to both the tiles and the hole.) 
\end{enumerate}
The following two chains are defined when $n$ is even. 
\begin{enumerate}
\item[4.] Parity-conditioned chain (PC): Choose a tile whose position 
has opposite parity to that of the hole, uniformly at random,
and then interchange it with the hole.

\item[5.] \ore chain (OR): The hole-conditioned chain, restricted to $\Omega$. 
That is, if $T_1 < T_2 < \cdots$ are the times when the 
hole-conditioned chain $X_t$ is in $\Omega$, then the 
\ore chain is $\{X_{T_j}: j \geq 1\}$. 
\end{enumerate}

The mixing time upper bound is a consequence of the following 
bound on the log Sobolev constant. 
\begin{theorem}
\label{logsobbound}
The log Sobolev constant $\alpha_{\rm Loyd} = \alpha_{\rm Loyd}(n)$ 
satisfies
\[
\alpha_{\rm Loyd} \geq D /(n^4 \log n),
\]
for a universal constant $D>0$.
\end{theorem}
Since the number of permutations on $V_n$ is 
$(n^2)! \leq (n^2)^{n^2}$, combining Theorem \ref{logsobbound}
with (\ref{logsob}) gives: 
\begin{corollary}
The mixing time for the Loyd process is $\tmix = \bigo(n^4 \log^2 n)$.
\end{corollary}
\begin{proofof}{Proof of Theorem \ref{logsobbound}}
The log Sobolev constant 
$\alpha_{\rm RT} = \alpha_{\rm RT}(n)$ 
for shuffling $n^2$ cards by random transpositions satisfies
\[
\alpha_{\rm RT} \geq c/(n^2 \log n),
\]
for a universal constant $c > 0$;
see \cite{ls, ly}.  

For the case when $n$ is even, Theorem \ref{logsobbound} 
follows from the following relations between log Sobolev
constants:
\[
C n^2 \alpha_{\rm Loyd} \geq \alpha_{\rm PC};
\;\;\;\;\;\;
(822) \alpha_{\rm PC} \geq \alpha_{\rm OR};
\;\;\;\;\;\;
2 \alpha_{\rm OR} \geq \alpha_{\rm HC};
\;\;\;\;\;\;
12 \alpha_{\rm HC} \geq \alpha_{\rm RT};
\]
which we prove below as Lemmas \ref{lpc}, \ref{pcor},
\ref{hcor} and \ref{hcrt}, respectively.

For the case when $n$ is odd, Theorem \ref{logsobbound} 
follows from the following relations between log Sobolev
constants:
\[
C n^2 \alpha_{\rm Loyd} \geq \alpha_{\rm HC};
\;\;\;\;\;\;\;\;\;\;\;\;
12 \alpha_{\rm HC} \geq \alpha_{\rm RT};
\]
which we prove below as Lemmas \ref{lhc} and \ref{hcrt}, respectively.

 The proof of Lemma \ref{hcrt} can be found in Section \ref{rthc}.
The proofs of 
Lemmas \ref{lpc}, \ref{pcor},
\ref{hcor} and \ref{lhc} can be found in Section \ref{others}.
\end{proofof}

\section{Comparison of hole-conditioned chain with random 
transpositions}
\label{rthc}
\begin{lemma}
\label{hcrt}
The log Sobolev constants $\alpha_{\rm RT}$ and 
$\alpha_{\rm HC}$ satisfy
\[
\alpha_{\rm RT} \leq 12 \alpha_{\rm HC}.
\] 
\end{lemma}
\begin{proof}
Let $G$ be the symmetric group on $V_n$ with 
the group operation  %on permutations $\pi$ and $\mu$ on $V_n$ 
defined by 
\[
\pi \mu = \mu \circ \pi.
\]
For permutations $\pi$ on $V_n$,
if we think of $\pi(j)$  as representing the label of the particle in position 
$j$, then we can view
shuffling by random transpositions (respectively, 
the hole-conditioned chain) as the random walk on $G$ driven 
by $\pt$ (respectively, $p$), where
\begin{eqnarray*}
\pt &=& \mbox{uniform distribution on permutations of the form 
$(i, j)$ with $i \neq j$ and $i, j \in V_n$};  \\
p &=& \mbox{uniform distribution on permutations of the form 
$(h, i)$ with $i \neq h$ and $i \in V_n$}.
\end{eqnarray*}
We compare the 
hole-conditioned chain with shuffling by random transpositions 
using Corollary \ref{lscomp}.
If $i < j$ we represent the permutation $(i,j)$ by $(h,i)(h,j)(h,i)$.
Let $m = n^2$. Consider the move $(h,i)$ in the support of $p$. 
Note that $(h,i)$ 
is in the representation of $m-1$ elements, each of the 
form $(i,j)$. Since $p\left((h,i)\right) = 1/(m-1)$ and 
$\pt\left((i,j)\right) = 
1/{m \choose 2}$, applying Corollary \ref{lscomp} and using the bounds 
$N(z,y) \leq 2$ and $|y| \leq 3$ gives $$A \leq 6(m-1)^2/{m \choose 2} 
< 12. $$ 
\end{proof}
\section{Comparisons involving the remaining chains}
\label{others}
The subsequent chains that we    analyze
are random walks  on a  different group. 
Note that the hole-conditioned chain, Loyd chain, and 
parity-conditioned chain all can be described as follows. 
At each step:
\begin{enumerate}
\item 
\label{y}
choose $y$ according to some distribution on $V_n$;
\item if the hole is in position $x$, 
interchange it with the tile in position $x + y$.
\end{enumerate} 
To see that these are random walks on a group, 
let $\ph = V_n \setminus (0,0)$ and 
note that 
a configuration
can be specified by an ordered pair 
$(x, f)$, where $x \in V_n$ is the position of the hole, and
$f: \ph \to \ph$, is the permutation defined by 
\[
f(z) = \left(\mbox{position of tile $z$} \right) - x.
\]  
(Thus $f$ gives the positions of the tiles relative to the hole;
note that $f$ maps tiles to positions, 
whereas for the permutations in Section \ref{rthc} it 
was the other way around.) 

Let $G$ be the group whose elements 
are $\{ (x, f): x \in V_n, \mbox{$f$ is a permutation on $\ph$} \}$
and with 
the group operation 
\[
(x, f) \cdot (y,g) = (x + y, g \circ f).
\]
Thus $G$ is the direct product
of $V_n$ and the symmetric group on $\ph$.  
%Note that the identity element is $(0 , \pi)$, where $\pi$ is the identity 
%permutation on $\ph$. 
For $y \in V_n$, the transition that 
translates hole by $y$ 
is right multiplication by the group element 
$(y, \pi_y)$, where
$\pi_y$ is the permutation defined by 
\begin{equation}
\label{ydef}
\pi_y(z) 
 = \left\{\begin{array}{ll}      
z - y             & \mbox{if $z \neq y$;} \\
-y           & \mbox{if $z = y$.} \\
\end{array}
\right.
\end{equation}
%Note that $(y, \pi_y)$ is configuration obtained by 
%performing the move $M(y)$ on the identity.
As an abuse of notation, we write $y$ for the move
$(y, \pi_y)$. We write
 $\uparrow, \downarrow, \rightarrow$, and  $\leftarrow$
for the moves $(0,1), (0, -1), (1,0)$, and $(-1,0)$, 
respectively. \\
\\
{\bf The \ore chain. } Note that $0$ is the identity 
element of $G$. 
If $n$ is even, and we define
\[
\Omega = \{ (x, f): \parity(x) = \parity(f) \},
\]
then $\Omega$ is the set of states reachable from $0$ 
in the Loyd chain.
Note that $\Omega$ is closed under products and inverses and 
hence is a subgroup of $G$.

It is not hard to show that the permutation $\pi_y$ defined 
in (\ref{ydef}) is odd unless $y = 0$. 
This implies
that
the  move $y$
is in 
$\Omega$ if and only if $y$ is odd or $0$.
We will call such moves
{\it good} and the other moves {\it bad.}
Note that the product of moves 
$y_1 y_2 \cdots y_m$ is in $\Omega$ if and only if an even number of 
the $y_i$ are bad.

The \ore chain is a random walk on $\Omega$ 
where each move is generated as 
follows:
\begin{enumerate}
\item Let $y_1, y_2, \dots$ be i.i.d.~moves of the hole-conditioned 
chain, and let 
$$T = \min\{m \geq 1: \mbox{an even number of
the moves $y_1, \dots, y_m$ are bad}\};$$ 

\item Let the move be $y_1 y_2 \cdots y_T$. 
\end{enumerate}

\subsection{Comparison of hole-conditioned chain with \ore chain}
%In this section we compare the log Sobolev constants of the hole-conditioned 
%chain
%and \ore chain.
%We will need the following lemma.
Note that the \ore chain is a ``sped up'' 
version of the hole-conditioned chain; this
suggests that its log Sobolev constant
should be comparable to that of the hole-conditioned 
chain. In this section we show that 
this is indeed the case. We will need the 
following lemma about the restriction of a Markov
chain and the Dirichlet form.
\begin{lemma}
\label{fflemma}
Let $P$ be a reversible Markov chain on a finite state space $V$.
Let $S \subset V$  
and let $\Pt$ be the restriction of $P$ to $V\setminus S$. 
Suppose that $f: V \to \R$ is harmonic on $S$ 
and let $\ft: V \setminus S \to \R$ be the restiction of $f$ to $V \setminus S$.
Then the Dirichlet forms $\eps$ and $\et$ satisfy
\[
\eps(f, f) \leq \et(\ft,\ft).
\]
\end{lemma}
\begin{proof}
The proof is by induction on $|S|$. 
For the base case $|S| = 1$, suppose that $S = \{x\}$. 
W.l.o.g, suppose that $p(x,x) = 0$. (Otherwise consider the 
chain $Q$ such that $q(x,x) = 0$ and  
$q(x,y) = {p(x,y) \over 1 - p(x,x)}$ for $y \neq x$.) Note that 
for $i,j \in V \setminus \{x\}$, we have
\begin{equation}
\label{ptid}
\pt(i,j) = p(i,j) + A(i,j),
\end{equation}
where $A(i,j) = p(i,x)p(x,j)$. Hence
\begin{equation}
\et(\ft, \ft) = 
\half \sum_{i,j \in V \setminus \{x\}}
\pit(i) p(i,j) (f(i) - f(j))^2 + 
\half \sum_{i,j \in V \setminus \{x\}}
\pit(i) A(i,j) (f(i) - f(j))^2, 
\end{equation}
and note that
\begin{equation}
\eps(f, f) = 
\half \sum_{i,j \in V \setminus \{x\}}
\pi(i) p(i,j) (f(i) - f(j))^2 + 
 \sum_{j \in V \setminus \{x\}}
\pi(x) p(x,j) (f(j) - f(x))^2. 
\end{equation}
Thus, since for every $i \in V \setminus \{x\}$ 
we have $\pit(i) \geq \pi(i)$, it is enough to show that
\begin{equation}
\label{st}
\sum_{i,j \in V \setminus \{x\}}
\pi(i) A(i,j) (f(i) - f(j))^2 \geq 
2 \sum_{j \in V \setminus \{x\}}
\pi(x) p(x,j) (f(j) - f(x))^2. 
\end{equation}
The lefthand side is 
\begin{eqnarray*}
\sum_{i,j \in V \setminus \{x\}}
\pi(i) p(i,x)p(x,j) (f(i) - f(j))^2 &=&
\pi(x) 
\sum_{i,j \in V \setminus \{x\}}
p(x,i)p(x,j) (f(i) - f(j))^2 \\
&=& \pi(x) \cdot 
2
\sum_{j \in V \setminus \{x\}}
p(x,j) (f(j) - f(x))^2,
\end{eqnarray*}
where the first line follows from detailed balance and 
the second line holds because $f$ is harmonic at $x$ and 
hence $f(x) = \sum_j p(x, j ) f(j)$. 
For the second line we are also  
using the fact that if $X$ and $Y$ are i.i.d.~random 
variables with mean $0$, then $\e(X - Y)^2 = 2 \var(X^2)$.
This verifies (\ref{st}).

Now suppose that the result holds when $|S| \leq k$ and 
suppose that $S = \{ x_1, \dots, x_{k+1}\}$ and $f$ is harmonic 
on $S$.  Let $\Pone$ be the restriction of $P$ to 
$V \setminus \{x_1\}$, let $\eone$ be the Dirichlet form 
with respect to $\Pone$ and let $\fone$ be the restriction 
of $f$ to $V \sm \{x_1\}$. Since $f$ is harmonic at $x_1$,
the induction hypothesis implies that 
\begin{equation}
\label{dag}
\eone( \fone, \fone) \geq \eps(f,f).
\end{equation}
Note that $\fone$ is harmonic with respect to $\Pone$ 
on $S \sm \{x_1\}$. Furthermore, the restriction of $\Pone$ 
to $(V \sm \{x_1\}) \sm (S \sm \{x_1\})$ is $\Pt$. Using the induction hypothesis
again, we get
\[
\et( \ft, \ft) \geq \eone(\fone, \fone).
\]
Combining this with (\ref{dag}) yields the lemma.
\end{proof}
\begin{lemma}
\label{hcor}
Suppose that $n$ is even. The log-Sobolev 
constants $\alphat$ and $\alpha$ of the  
\ore 
and 
hole-conditioned 
chain, respectively, satisfy
\[
\alphat \geq \half \alpha.
\]
\end{lemma}
\begin{proof}
If $\pi$ (respectively, $\pit$)
is the stationary disribution for the hole-conditioned 
chain (respectively, \ore chain), 
then $\pit(x) = 2 \pi(x)$ for $x \in \Omega$. 
Let $\ft: \Omega \to \R$ be such that $\ent_{\pit}(\ft^2) \neq 0$.
Let $f$ be the harmonic extension of $\ft$ to $S$. We shall show that
\begin{equation}
\label{ls}
{ \et(\ft, \ft) \over \ent_\pit( f^2) } 
\geq 
{ \eps(f, f) \over 2 \ent_\pi( f^2) } .
\end{equation}
We compare numerators and then denominators. 
Since $f$ is harmonic on $S$, Lemma {\ref{fflemma}} implies that
\[
\et(\ft, \ft) \geq \eps(f,f).
\]
Next we compare denominators.
We claim that 
$2 \entpi(f^2)
\geq 
\entpit(\ft^2)$.  
To see this, let
$g = f^2$ and let $\pih$ be the 
uniform distribution over $\Omega^c$.
Then we can write $\entpi(f^2)$ as
\begin{eqnarray*}
\half \e_\pit \left[ g \log {g \over \e_{\pit}(g)} \right]  +   \\
\half \e_\pih \left[ g \log {g \over \e_{\pih}(g)} \right]  +   \\
\left[
\half \e_\pit ( g) \log {\e_{\pit}(g) \over \e_{\pi}(g)} 
                   +
\half \e_\pih (g ) \log {\e_{\pih}(g) \over \e_{\pi}(g)} 
\right]  .
\end{eqnarray*}
Since for all constants $a$ the function 
$x \mapsto x \log (ax)$ is convex, the expressions on the 
second and third lines are nonnegative.
It follows that
\begin{eqnarray*}
\entpi(f^2) 
&\geq& \half \e_{\pit} 
\Bigl( g \log {g \over \e_{\pit}(g)} 
\Bigr).
\end{eqnarray*}
The claim follows since
the quantity on the right-hand side is $\half \entpit(\ft^2)$
since $g = \ft^2$ on $\Omega$. 
This proves the lemma since $f$ is arbitrary. 
\end{proof}

\subsection{Comparison of parity-conditioned   chain to \ore chain}
\begin{lemma}
\label{pcor}
Suppose that $n$ is even. Then the 
log Sobolev constants $\alpha_{\rm PC}$ and 
$\alpha_{\rm OR}$ satisfy
\[
\alpha_{\rm OR} \leq (882) \alpha_{\rm PC}.
\] 
\end{lemma}
\begin{proof}
In order to compare the \ore chain 
with the 
parity-conditioned chain 
we intoduce an 
intermediate chain, which we denote \bgb. 
A move of the BGB chain is a concatenation consisting of between 
$1$ and $3$ moves of the HC chain, generated as follows. 
Let $b_1$ and $b_2$ be uniform 
random bad moves, let $g$ be a uniform random good move. The BGB move 
is 
\[
x 
 = \left\{\begin{array}{ll}      
 g          & \mbox{with probability $1/3$;} \\
 b_1 b_2          & \mbox{with probability $1/3$;} \\
 b_1 g b_2         & \mbox{with probability $1/3$.} \\
\end{array}
\right.
\]

We shall use Corollary \ref{lscomp} twice, first to compare 
\ore with BGB, then to compare BGB with PC. \\
\\
{\bf Comparison of  \ore chain with \bgb chain.}
We need to show how 
to represent moves of the \ore chain using BGB moves. 
Consider a move $y$ of the \ore chain. Then $y$ is  
of the form $g$, $b_1 b_2$ or $b_1 g_1 g_2 \cdots g_k b_2$, where 
we write $b$'s for bad moves and $g$'s for good moves. If 
$y = g$  (respectively, $y = b_1b_2$)
then we can represent it as $g$ (respectively, $b_1b_2$),
since this is also a \bgb move. Suppose now that
\[
y = b_1 g_1 g_2 \cdots g_k b_2.
\]
In this case we represent it as $z_1 \cdots z_k$, where the 
$z_i$ are defined by
\[
\underbrace{(b_1 g_1 B_1)}_{z_1}
\underbrace{(B_1 g_2 B_2)}_{z_2}
\underbrace{(B_2 g_3 B_3)}_{z_3}
\cdots
\underbrace{(B_{k-1} g_k b_2)}_{z_k},
\]
for uniform random bad moves $B_1, \dots, B_{k-1}$. 

We apply Corollary \ref{lscomp}, 
letting $\pt$ (respectively, $p$) be the measure corresponding 
to the \ore chain (respectively, \bgb chain).
We need to bound the quantity
$A = \max_z A(z)$, where
\[
A(z) =  {1 \over p(z)} \sum_y \pt(y) N(y,z) |y|.
\]
Let $m = n^2$.
If $z = g$ 
%for some good move $g$, 
then $z$ is used only in the representation of $g$ itself, and 
hence $A(z) = {\pt(g) \over p(z)} = {3m \over 2(m-1)}$. 
Similarly, if $z = b_1 b_2$ then $A(z) = {\pt(b_1 b_2) \over p(b_1 b_2)}
= { 3(m-2)^2 \over 4 (m-1)^2}$. 

It remains to check the case when $z$ is of the form $b_1 g b_2$. Note that 
$A(z)$ can be written as ${1 \over p(z)} \e \left( N(Y, z) |Y| \right),$ where
$Y$ is a random move 
chosen from $\pt$. Define the random variable $K$ by 
\[
K
 = \left\{\begin{array}{ll}      
$k$         & 
\mbox{if $Y =  b_1 g_1 g_2 \cdots g_k b_2$;} \\
0          & \mbox{if $Y$ is of the form $g$ or $b_1 b_2$.} \\
\end{array}
\right.
\]
Note that $\P(K = k) = \left( \half \right)^{k + 2}$ for $k \geq 1$.
Furthermore, conditional on $K = k$, the distributions of 
$Z_1, \dots, Z_k$ are uniform over moves of the form $b_1gb_2$. It follows
that 
\begin{eqnarray*}
\e \Bigl( N(Y, z) |Y| \Given K = k \Bigr) &=& 
k \e \Bigl( N(Y,z) \Given K = k \Bigr) \\
&=& k \; {k \over |S|},
\end{eqnarray*}
where $S$ is the set of moves of the form $b_1 g b_2$. It follows that 
\begin{eqnarray*}
\e \left( N(Y, z) |Y| \right) &=& \sum_{k \geq 1} \left( \half \right)^{k+2}
{ k^2 \over |S| } \\
&=& {3 \over 2 |S|}.
\end{eqnarray*}
Since $p(z) = {1 \over 3 |S|}$, we have $A(z) = 
{1 \over p(z)} \e \left( N(Y, z) |Y| \right) = 9/2.$ 
Hence $A = 9/2$ as well,
and hence
\begin{equation}
\label{orbgb}
\alpha_{\rm OR} \leq {9 \over 2} \alpha_{\rm BGB}.
\end{equation} 
{\bf Comparison of \bgb chain with PC chain.}
We need to show how to represent 
a \bgb move with PC moves. 
Consider a move $y$ of the \bgb chain. 
If $y = g$ then we represent it as $g$ itself.  
To handle moves of the form $b_1 b_2$ and 
$b_1 g b_2$, we 
first note that if 
$e_1, e_2 \in \ph$ are even and $o \in \ph$ is odd, then 
we can represent the \bgb move $e_1 o e_2$ as
\begin{equation}
\label{bgbrep}
(e_1 + o) (-o) (o + e_2)(-e_1 - o - e_2)(e_1 + o)(-o)(o + e_2).
\end{equation}
Note that the moves in
(\ref{bgbrep}) are moves of the PC chain, 
since the corresponding  elements of $V_n$ are odd. 
If $y$ is of the form $b_1 g b_2$, we can 
represent it with PC moves using (\ref{bgbrep}). If $y$ is 
of the form $b_1 b_2$, we first give it the intermediate
representation $(b_1 G B) (B G  b_2)$, where $B$ and $G$ are uniform
random bad and good moves, respectively, and then represent 
both the $b_1 G B$ and $BG b_2$ using (\ref{bgbrep}). Note that the 
maximum length
of the representation of any $y$ is $14$

We apply Corollary \ref{lscomp} again, this time letting $\pt$ 
(respectively, $p$)
be the measure corresponding 
to the \bgb chain (respectively, PC chain). 
We need to bound the quantity
\[
A = \max_z {1 \over p(z)} \e\left( N(Y,z) |Y| \right),
\]
where $Y$ is chosen according to $\pt$. Let 
$Y = Z_1 \cdots Z_K$ be the representation of $Y$. Note that
for all $k \in \{1, 7, 14\}$ 
 the 
conditional distribution of $Z_1, \dots, Z_k$, given $|Y| = k$ is 
uniform over the set of PC moves. It follows that
for every PC move $z$ we have
\begin{eqnarray*}
\e \Bigl( N(Y, z) |Y| \Given |Y| = k \Bigr) &=& 
k \e \Bigl( N(Y,z) \Given K = k \Bigr) \\
&=& k \; {k \over |PC|}
\end{eqnarray*}
where we write $|PC|$ for the number of PC moves. It follows that, for 
any PC move $z$, we have
\begin{eqnarray*}
\e \left( N(Y, z) |Y| \right) &=& {1 \over |PC|} 
\e\left(|Y|^2 \right) \\
&\leq& { 196 \over |PC|},
\end{eqnarray*}
where the last line holds because $|Y| \leq 14$.
Since $p$ is the uniform distribution 
over PC moves, we have $p(z) = {1 \over |PC|}$,
and hence
${1 \over p(z)} \e \left( N(Y, z) |Y| \right) \leq 196.$ 
Hence $A \leq 196$, which implies that
\begin{equation}
\alpha_{\rm BGB} \leq (196) \alpha_{{\rm PC}}.
\end{equation}
Combining this with (\ref{orbgb}) yields the lemma.
\end{proof}

\subsection{Comparisons of  parity-conditioned 
and hole-conditioned chains with Loyd chain}
\begin{lemma}
\label{lpc}
Suppose that $n$ is even. 
Then the log Sobolev constants $\alpha_{\rm Loyd}$ and 
$\alpha_{\rm PC}$ satisfy
\[
C n^2 \alpha_{\rm Loyd} \geq \alpha_{\rm PC},
\] 
for a universal constant $C$.
\end{lemma}
\begin{proof}
In order to  apply Corollary \ref{lscomp}, 
we need to show how to represent any move of the PC chain 
using moves of the Loyd chain. 
We will actually show how to represent PC moves 
using a different Markov chain,
which we call 
{\it near Loyd} (NL). In the NL chain, each move is 
a move $x$ of the PC chain with $x$ conditioned to satisfy
$|x_1| + |x_2| \in \{1,3\}$, where for $u \in \z_n$ 
we define $|u| = \min(u, n-u)$. That is, each step 
of the NL chain swaps the hole with tile at $L^1$-distance
$1$ or $3$ away from it. 
A representation using NL moves is sufficient 
because
any NL move can be represented using a 
bounded number of Loyd moves:
if the $L^1$-distance between 
the hole and tile $T$ is at most $3$, then there is 
a $3 \times 3$ square grid that contains both the hole and tile $T$,
and the fifteen puzzle is solvable in a $3 \times 3$ grid. 

We  now 
show how to represent a PC move with NL moves. 
There are three cases to consider.\\
\\
{\bf Case 1: swapping the hole with a tile one row higher. } We first 
consider the case where the move $y = (y_1, y_2)$ is such that
$y_2 = 1$. That is, the move swaps the hole with a tile one row higher. 

Suppose that tile $T$ is located in the row 
immediately above the hole. 
To swap the hole with tile $T$, leaving everything
else the same,  
do the following:
\begin{enumerate}
\item repeat: $\uparrow, \rightarrow, \downarrow, \rightarrow$, 
until the hole is swapped with $T$. 

\item do $\leftarrow, \downarrow, \rightarrow$ once.

\item repeat:  $\uparrow, \leftarrow, \leftarrow, 
\downarrow, \rightarrow$, until $T$ is in the position that the 
hole initially occupied. 

\item repeat: alternate $\rightarrow, \uparrow$
and
$\rightarrow,\downarrow$ until the hole is in the position 
initially occupied by $T$. 

\end{enumerate}
\noindent
Note that each move here is actually a move of the Loyd chain.

Figures \ref{fig1}--\ref{fig5} show an application 
of the algorithm. In this example, the hole is swapped with 
the tile of label $5$. 

\begin{figure}[H]
\centering
\begin{tikzpicture} 
  \foreach \y in {0,2,4}
  \draw [fill = gray] (\y, 0) rectangle (\y + 1, 1);
  \foreach \y in {1,3}
  \draw [fill = gray] (\y, 1) rectangle (\y + 1, 2);
  \foreach \y in {1,3}
  \draw [fill = white] (\y, 0) rectangle (\y + 1, 1);
  \foreach \y in {0,2,4}
  \draw [fill = white] (\y, 1) rectangle (\y + 1, 2);

\node at (.5,1.5) {1};
\node at (1.5,1.5) {2};
\node at (2.5,1.5) {3};
\node at (3.5,1.5) {4};
\node at (4.5,1.5) {5};

\node at (.5,.5) {h};
\node at (1.5,.5) {9};
\node at (2.5,.5) {8};
\node at (3.5,.5) {7};
\node at (4.5,.5) {6};
\end{tikzpicture}
\caption{Initial configuration}\label{fig1}
\end{figure}

\begin{figure}[H]
\centering
\begin{tikzpicture} 
  \foreach \y in {0,2,4}
  \draw [fill = gray] (\y, 0) rectangle (\y + 1, 1);
  \foreach \y in {1,3}
  \draw [fill = gray] (\y, 1) rectangle (\y + 1, 2);
  \foreach \y in {1,3}
  \draw [fill = white] (\y, 0) rectangle (\y + 1, 1);
  \foreach \y in {0,2,4}
  \draw [fill = white] (\y, 1) rectangle (\y + 1, 2);

\node at (.5,1.5) {2};
\node at (1.5,1.5) {9};
\node at (2.5,1.5) {4};
\node at (3.5,1.5) {7};
\node at (4.5,1.5) {h};

\node at (.5,.5) {1};
\node at (1.5,.5) {8};
\node at (2.5,.5) {3};
\node at (3.5,.5) {6};
\node at (4.5,.5) {5};
\end{tikzpicture}
\caption{After step 1}\label{fig2}
\end{figure}

\begin{figure}[H]
\centering
\begin{tikzpicture} 
  \foreach \y in {0,2,4}
  \draw [fill = gray] (\y, 0) rectangle (\y + 1, 1);
  \foreach \y in {1,3}
  \draw [fill = gray] (\y, 1) rectangle (\y + 1, 2);
  \foreach \y in {1,3}
  \draw [fill = white] (\y, 0) rectangle (\y + 1, 1);
  \foreach \y in {0,2,4}
  \draw [fill = white] (\y, 1) rectangle (\y + 1, 2);

\node at (.5,1.5) {2};
\node at (1.5,1.5) {9};
\node at (2.5,1.5) {4};
\node at (3.5,1.5) {6};
\node at (4.5,1.5) {7};

\node at (.5,.5) {1};
\node at (1.5,.5) {8};
\node at (2.5,.5) {3};
\node at (3.5,.5) {5};
\node at (4.5,.5) {h};
\end{tikzpicture}
\caption{After step 2}\label{fig3}
\end{figure}

\begin{figure}[H]
\centering
\begin{tikzpicture} 
  \foreach \y in {0,2,4}
  \draw [fill = gray] (\y, 0) rectangle (\y + 1, 1);
  \foreach \y in {1,3}
  \draw [fill = gray] (\y, 1) rectangle (\y + 1, 2);
  \foreach \y in {1,3}
  \draw [fill = white] (\y, 0) rectangle (\y + 1, 1);
  \foreach \y in {0,2,4}
  \draw [fill = white] (\y, 1) rectangle (\y + 1, 2);

\node at (.5,1.5) {1};
\node at (1.5,1.5) {2};
\node at (2.5,1.5) {8};
\node at (3.5,1.5) {3};
\node at (4.5,1.5) {6};

\node at (.5,.5) {5};
\node at (1.5,.5) {h};
\node at (2.5,.5) {9};
\node at (3.5,.5) {4};
\node at (4.5,.5) {7};
\end{tikzpicture}
\caption{After step 3}\label{fig4}
\end{figure}

\begin{figure}[H]
\centering
\begin{tikzpicture} 
  \foreach \y in {0,2,4}
  \draw [fill = gray] (\y, 0) rectangle (\y + 1, 1);
  \foreach \y in {1,3}
  \draw [fill = gray] (\y, 1) rectangle (\y + 1, 2);
  \foreach \y in {1,3}
  \draw [fill = white] (\y, 0) rectangle (\y + 1, 1);
  \foreach \y in {0,2,4}
  \draw [fill = white] (\y, 1) rectangle (\y + 1, 2);

\node at (.5,1.5) {1};
\node at (1.5,1.5) {2};
\node at (2.5,1.5) {3};
\node at (3.5,1.5) {4};
\node at (4.5,1.5) {h};

\node at (.5,.5) {5};
\node at (1.5,.5) {9};
\node at (2.5,.5) {8};
\node at (3.5,.5) {7};
\node at (4.5,.5) {6};
\end{tikzpicture}
\caption{After step 4 (final position)}\label{fig5}
\end{figure}

{\bf Case 2: swapping the hole with a tile on the same row.}
Let $\config$ be a configuration in which the hole and 
tile $T$ are on the same row. 
To swap the hole with tile $T$:
Choose a tile $T'$ on the row 
one step higher such that $T$ and $T'$ share one edge. 
Let $\config'$ be the 
configuration obtained from $\config$ by 
interchanging $T$ and $T'$. 
Let $f$ be the permutation 
on $V_n$ that transposes the positions of tiles $T$ and $T'$ in 
configuration $\config$.
Since in $\config'$ tile $T$ is one row higher than 
the hole, we can use the algorithm for Case 1 
to swap the hole and $T$ starting from configuration 
$\config'$.
Let $l_k$ be the label of the tile swapped with 
the hole in the $k$th step when performing this algorithm. 
To swap the hole with tile $T$ starting from configuration $\config$,
we 
use the sequence of moves defined by the same 
label sequence $(l_1, l_2, \dots)$.
Note that
if a tile is in position $x$ after $k$ steps of the algorithm
starting from $\config'$, then it is in position $f(x)$ after 
$k$ steps of the algorithm starting from $\config$.  
Since the algoithm for Case 1 performs only Loyd moves,
the resulting algorithm for $\config$ 
swaps the hole with tiles at a distance either $1$ or $3$ 
from it, that is, it performs only NL moves. 
%(Here {\it distance} means $L^1$-distance.)

{\bf Case 3: swapping the hole with a tile not on the same row or
next row up.}
Now we consider the situation not covered in Case 1 or Case 2. 
The cases where tile $T$ is 
in the column 
to the immediate right of the hole or
in the same column as the hole are similar to above,
so assume neither of these situations hold,
as in 
Figure \ref{figh3}.
Let $\config$ be the configuration shown in Figure \ref{figh3}
and let $\config'$ be the configuration shown in Figure
\ref{figh4}. Let $f$ be the bijection from locations
in $\config$ to locations in $\config'$ that leaves the 
horizontal part unchanged and rotates and inverts the vertical 
part (which consists of locations in the column of $T$ and in the 
column one unit to the left of $T$) 
so that the location of tile $T$ is sent to 
the row second from the bottom.
Since in $\config'$ tile $T$ is in the row second from the bottom, we
can use the algorithm for Case 1 to swap the hole 
with tile $T$, using only Loyd moves, 
starting from configuration $\config'$. 
As before, we can use the labels of the tiles moved at each 
step to define an algorithm starting from configuration $\config$.
Note that if positions $x$ and $y$ are adjacent in $\config'$
then $f^{-1}(x)$ and $f^{-1}(y)$ are at distance 
$1$ or $3$ from each other in $\config$.
It follows that the algorithm for configuration $\config$ 
swaps the hole with tiles at a distance $1$ or $3$ 
from it, that is, performs only NL moves. 

Note that the maximum length of
the representation of a PC move using NL moves 
is at most $B n$, for a universal constant $B$.
This also applies 
to the resulting representation using Loyd moves.
 
We apply Corollary \ref{lscomp} again, 
this time letting $\pt$ (respectively, $p$)
be the measure corresponding 
to the PC chain (respectively, Loyd chain). 
We need to bound the quantity
\[
A = \max_{z}
{1 \over p(z)} \e\left( N(Y,z) |Y| \right),
\]
where $Y$ is chosen according to $\pt$. Since $N(Y, z) \leq |Y| \leq Bn$, 
and $p(z) = 1/4$ 
for $z \in \{\leftarrow, \rightarrow, \uparrow, \downarrow\}$, 
we have
\[
A \leq C n^2
\]
for a universal constant $C$, and hence
\begin{equation}
\label{pcnl}
\alpha_{\rm PC} \leq C n^2 \alpha_{\rm Loyd}.
\end{equation}

\begin{figure}[H]
\begin{tikzpicture} 
  \foreach \y in {0,2,4}
  \draw [fill = gray] (\y, 0) rectangle (\y + 1, 1);
  \foreach \y in {1,3}
  \draw [fill = gray] (\y, 1) rectangle (\y + 1, 2);
  \foreach \y in {1,3}
  \draw [fill = white] (\y, 0) rectangle (\y + 1, 1);
  \foreach \y in {0,2,4}
  \draw [fill = white] (\y, 1) rectangle (\y + 1, 2);

\foreach \y in {2,4}
\draw (5, \y) rectangle (6, \y + 1);

\foreach \y in {3, 5}
\draw (6, \y) rectangle (7, \y + 1);

\foreach \y in {3,5}
\draw [fill = gray] (5, \y) rectangle (6, \y + 1);

\foreach \y in {2,4}
\draw [fill = gray] (6, \y) rectangle (7, \y + 1);

\node at (6.5, 5.5) {$T$};
\node at (6.5, 4.5) {15};
\node at (6.5, 3.5) {13};
\node at (6.5, 2.5) {11};

\node at (5.5, 5.5) {16};
\node at (5.5, 4.5) {14};
\node at (5.5, 3.5) {12};
\node at (5.5, 2.5) {10};

\node at (.5,.5) {h};
\node at (1.5,.5) {1};
\node at (2.5,.5) {2};
\node at (3.5,.5) {3};
\node at (4.5,.5) {4};

\node at (.5,1.5) {5};
\node at (1.5,1.5) {6};
\node at (2.5,1.5) {7};
\node at (3.5,1.5) {8};
\node at (4.5,1.5) {9};

\end{tikzpicture}
\caption{}\label{figh3}
\end{figure}

\begin{figure}[H]
\begin{tikzpicture} 
  \foreach \y in {0,2,4,6,8}
  \draw [fill = gray] (\y, 0) rectangle (\y + 1, 1);
  \foreach \y in {1,3,5,7}
  \draw [fill = gray] (\y, 1) rectangle (\y + 1, 2);
  \foreach \y in {1,3,5,7}
  \draw [fill = white] (\y, 0) rectangle (\y + 1, 1);
  \foreach \y in {0,2,4,6,8}
  \draw [fill = white] (\y, 1) rectangle (\y + 1, 2);
\node at (.5,.5) {h};

\node at (8.5, 1.5) {$T$};
\node at (7.5, 1.5) {15};
\node at (6.5, 1.5) {13};
\node at (5.5, 1.5) {11};

\node at (8.5, .5) {16};
\node at (7.5, .5) {14};
\node at (6.5, .5) {12};
\node at (5.5, .5) {10};

\node at (.5,.5) {h};
\node at (1.5,.5) {1};
\node at (2.5,.5) {2};
\node at (3.5,.5) {3};
\node at (4.5,.5) {4};

\node at (.5,1.5) {5};
\node at (1.5,1.5) {6};
\node at (2.5,1.5) {7};
\node at (3.5,1.5) {8};
\node at (4.5,1.5) {9};

\node at (8.5, 1.5) {$T$};
\draw [very thick] (5,0) -- (5,2);
\end{tikzpicture}
\caption{}\label{figh4}
\end{figure}

\end{proof}

\begin{lemma}
\label{lhc}
Suppose that $n$ is odd. 
Then the log Sobolev constants $\alpha_{\rm Loyd}$ and 
$\alpha_{\rm HC}$ satisfy
\[
C n^2 \alpha_{\rm Loyd} \geq \alpha_{\rm HC},
\] 
for a universal constant $C$.
\end{lemma}
\begin{proof}
The proof follows the proof of Lemma \ref{lpc} closely.
We will show how to represent any HC move using Loyd moves. 
Consider a move $y = (y_1, y_2)$ of the HC chain. If $y$ is odd then it is 
also a PC move and hence we can represent it using Loyd  moves
using the algorithm from the proof of Lemma \ref{lpc}.
If $y$ is even, then $(-y_1, y_2)$ is odd, and
we can represent $y$ using Loyd moves as follows:
we  perform the algorithm from the proof of Lemma \ref{lpc} to
swap the hole
with the tile in position $(-y_1, y_2)$, but we interchange the
roles of $\leftarrow$
and $\rightarrow$ moves. The resulting algorithm will swap the hole with 
the tile in position $(y_1, y_2) = y$. 

 We have shown that any HC move can be 
represented by $\bigo(n)$ Loyd moves, so the theorem follows 
by calculations similar to those leading up 
to equation (\ref{pcnl}).
\end{proof}

\section{Lower bound}
\label{lb}
In this section we prove a lower bound on the order of 
$n^4 \log n$ for the mixing time of the Loyd chain.
For the lower bound, a key fact is that
if we look at a tile at times 
when the hole is immediately to its right, the 
$x$-coordinate is doing a random walk on $\Z_n$.
More precisely, let $\{L_t : t \geq 0 \}$ be a Loyd process. 
We write $L_t(s)$ for the position of tile $s$ at time $t$. 
For a configuration $L$ and tile $s$ let 
$X(L,s)$ denote the $x$-coordinate of tile $s$ in configuration $L$,
and define $X_t(s) := X(L_t, s)$.  
Define $\taut_1(s), \taut_2(s), \dots$ inductively as follows.
Let $\taut_1(s)$ 
be the first time $t$ 
such that the hole is immediately to the right of tile $s$ at time $t$,
and for $k > 1$, let $\taut_k(s)$ be the first time $t > \taut_{k-1}(s)$
such that the hole is immediately to the right of tile $s$ at time $t$. 
The process
$\{X_{\taut_k(s)}(s) : k \geq 0\}$ is a symmetric random walk on $\z_n$,
which we shall call the {\it $s$ random walk}.
To see this, note that if $m_1 m_2 \cdots m_l$ is a sequence of moves
between times $\taut_1(s)$ and $\tau_2(s)$ that changes 
$X_t(s)$ from $x$ to $x + 1$ ($\bmod\;  n$), then the 
sequence of moves $m_l^{-1}, m_{l-1}^{-1}, \dots, m_1^{-1}$,
which occurs with the same probability, would 
change $x$ to $x-1$ ($\bmod \;n$) over the same time interval.
Note that each step of the $s$ random walk has a positive holding probability,
which is the probability that between times $\taut_k(s)$ 
and $\taut_{k+1}(s)$ 
the 
value of $X_t(s)$ does not change. 

Recall that for simple symmetric random walk
on a cycle of length $n$,
$f(x) = \cos { 2 \pi x \over n} $ is an eigenfunction 
with 
corresponding eigenvalue $\cos { 2 \pi\over n} $. 
Thus $f$ is an eigenfunction for the $s$ random walk as well.
Since the $s$ random walk has a holding probability the corresponding 
eigenvalue $\lambda > \cos { 2 \pi\over n} $. 

%It will be convenient to assume that the hole is initially somewhere 
%in the right half of the board. 
%Suppose that the hole is initially in position $(n/2 + 2,0)$.
The rough idea behind the lower 
bound will be to show that the tiles that start 
with an $x$-coordinate close to $0$ will 
tend to stay that way 
if the number of random walk steps is too low. 
Let $S$ be the set of tiles $s$
such that $f(X_0(s)) > 1/2$.
and suppose that 
the hole is not initially adjacent to any tile in $S$.  
Let $\mu$  be large enough so that
\begin{equation}
\label{mudef}
\left( \cos {2\pi \over n} \right)^{n^2} \geq e^{-\mu},
\end{equation}
for all $n \geq 2$. (Such a $\mu$ exists because 
$\cos x$ has the power series expansion $1 - {x^2 \over 2!}
+ {x^4 \over 4!} - \cdots$.) Next,  define
\begin{equation}  
\label{epsdef}
\epsilon = \sfrac{1}{8} \mu^{-1}; \;\;\;\;\;\;\;\;\;
\th = \lfloor 1 + \epsilon n^2 \log n \rfloor 
;\;\;\; \;\;\;\;\;\; T = (n^2 - 1) \th. 
\end{equation}

Since there are $n^2 - 1$ tiles, we can think of the quantity
$\th$ as the typical number of
times that the hole has been to the immediate right of any given tile % $s$
%(and hence the number of $s$ random walk steps) 
if the Loyd 
process has made $T$ steps. 

We shall bound the 
mixing time from below by $T$, 
which is on the order of $n^4 \log n$. 
We accomplish this using as a 
{\it distinguishing statistic} the random 
variable $\wdist$ defined by
\[
\wdist = \sum_{s \in S} f(X_T(s)).
\]
%We will show that $\wdist$ is likely to be too large.

Let $k = |S|$ and let $W$ be the sum of $k$ samples without replacement
from a population consisting of values of $\cos {2 \pi x \over n}$
for vertices $(x, y) \in V_n$.
The lower 
bound follows from Lemmas \ref{lemmaa} and \ref{lemmab} 
below, which together imply that $\Vert   \wdist - W \Vert_{TV} 
\to 1$ as $n \to \infty$.  
In the statements of 
Lemmas \ref{lemmaa} and \ref{lemmab},  
the random variables 
depend implicitly on the parameter $n$ of the
Loyd process.

\begin{alemma}
\label{lemmaa}
There is a universal constant $c > 0$ such that 
\[
\P( \wdist > c n^{15/8} ) \to 1,
\]
as $n \to \infty$. 
\end{alemma}

\begin{alemma}
\label{lemmab}
For any $c > 0$ we have
\[
\P(W > c n^{15/8} ) \to 0,
\]
as $n \to \infty$. 
\end{alemma}

\begin{theorem}
\label{mainlower}
Let $L_t$ be the Loyd process on $G_n$, and let $\pi$ be the 
stationary distribution. There is a universal constant $c > 0$ 
such that for any $\epsilon > 0$,
when $n$ is sufficiently large, we have
\[
\tmix(\epsilon) > c n^4 \log n.
\]
\end{theorem}
\begin{proof}
Lemmas \ref{lemmaa} and \ref{lemmab} 
together imply that $\Vert   \wdist - W \Vert_{TV} 
\to 1$ as $n \to \infty$.  This implies the Theorem 
since $\wdist$ is measurable with respect to $L_T$ and
$T \geq c n^4 \log n$ for a universal constant $c > 0$. 
\end{proof}

We prove Lemma \ref{lemmaa} in subsection \ref{lap}.
Lemma \ref{lemmab} is a straightforward consequence of 
Hoeffding's bounds for sampling without replacement in \cite{H}, 
which we recall now.
\begin{theorem}
\label{hoeffdinga}
Let $X_1, \dots, X_k$ be samples, without replacement, from a population
whose values are in the interval $[a,b]$, and suppose that
the population mean $\e(X_1) = 0$. 

Then for $\alpha>0$,
\begin{equation}
\label{hoeffding}
\P\left( \sum_{i=1}^k X_i \geq \alpha\right) \leq e^{-2 \alpha^2/k(b-a)^2}.
\end{equation}
%The bound (\ref{hoeffding}) applies whether the sampling
%is done with or without replacement.
\end{theorem}

\begin{proofof}{Proof of Lemma \ref{lemmab}}
Let $k = |S|$. Applying Theorem \ref{hoeffdinga} to 
$k$ 
samples 
from a population  consisting of values of $\cos {2 \pi x \over n}$
for vertices $(x, y) \in V_n$ gives 
\begin{equation}
\label{hoeffup}
\P\left( \sum_{i = 1}^k X_i \geq n^{15/8} \right) \leq 
\exp\left(-n^{15/4} / 2 k \right).
\end{equation}
Since $k \leq n^2$, the quantity (\ref{hoeffup}) converges to $0$ 
as $n \to \infty$.
\end{proofof}

\subsection{Proof of Lemma \ref{lemmaa}}
\label{lap}
For $s \in S$, let $\nt_t(s)$ be the number of 
times that the hole has been to the immediate right of tile $s$, 
up to time $t$. 
Note that for all $t$, if $\nt_t(s) > 0$ then 
\[
\taut_{\nt_t(s)}(s) \leq t < \taut_{\nt_t(s) + 1}(s).
\] 
Recall that $f(x)  = \cos {2 \pi x \over n}$.
It follows that $f'(x) = - {2 \pi \over n} 
\sin {2 \pi x \over n}$
and hence $|f'(x)| \leq {2 \pi \over n}$ for all $x$.
Thus 
the mean value theorem implies that
for every $x$ and $k$ we have
\begin{equation}
\label{mvt}
|f( x + k) - f(x) | \leq {2 \pi |k| \over n}. 
\end{equation}

%Define $X_\fin(s) := X_{\tau_\th(s)}(s)$.
%, that is,
%the value of $X_t(s)$ when $N_t(s) = \th$.  
We will prove Lemma \ref{lemmaa} by approximating $\wdist$ by the
random variable $Z := \sum_{s \in S} X_{\tau_\th(s)}(s)$.
The random variable $Z$ is easier to analyze
than $\wdist$
(but couldn't be used as a distinguishing statistic itself
because it is not measurable with respect to $L_t$ for any $t$).
%Define $Z = \sum_{s \in S} f( X_\fin(s))$. 
For the proof of Lemma \ref{lemmaa}
we will 
need the following propositions.
\begin{proposition} 
\label{lemma1}
For any $b > 0$ we have
\[
\P \left( 
\left| \sum_{s \in S} 
f( X_{\tau_{\nt_T}(s)}(s)) - Z \right|
> b n^{7/4}  \right) \to 0
\]
as $n \to \infty$. 
\end{proposition} 
\begin{proposition}
\label{lemma2}
For any $b > 0$ we have
\[
\P \left(  \left|
Z - 
\e(Z)  \right| > b n^{7/4}
\right) \to 0
\]
as $n \to \infty$.
\end{proposition}
We defer the proofs of Propositions \ref{lemma1} and 
\ref{lemma2} to subsection \ref{l12proofs}. 
We now give a proof of Lemma \ref{lemmaa}, assuming
Propositions \ref{lemma1} and \ref{lemma2}.

\begin{proofof}{Proof of Lemma \ref{lemmaa}}
Recall that $\wdist = \sum_{s \in S} f(X_T(s))$. 
Since for any tile $s \in S$ we have
$| X_{\tau_{\nt_T}(s)}(s) - X_T(s)| \leq 1$,
it follows that 
$| f(X_{\tau_{\nt_T}(s)}(s)) - f(X_T(s))| \leq {2 \pi \over n}$,
by (\ref{mvt}).
Thus
\begin{eqnarray}
\left| \wdist - \sum_{s \in S} 
f( X_{\tau_{\nt_T}(s)}(s)) \right| &\leq& \sum_{s \in S}
\left| f( X_{\tau_{\nt_T}(s)}(s)) - f(X_T(s)) \right|   \\
\label{close}
&\leq& 2 \pi n,
\end{eqnarray}
where the last line holds because $|S| \leq n^2$. 
%A similar argument shows that since $|X_{\taut_1(s)}(s) - X_0(s)| \leq 1$,
%we must also have
%\begin{equation}
%\label{startclose}
%\left| \sum_{s \in S} 
%f( X_{\tau_1(s)}(s)) 
%- 
%\sum_{s \in S}  f( X_0(s)) 
%\right| \leq 
%2 \pi n.
%\end{equation}
%Define $Z := \sum_{s \in S} f(X_\fin(s))$. 
The main remaining step of the proof is to compute
$\e (Z)$.  We claim that
$\e(Z) \geq c n^{15/8}$, for a universal constant $c$. 
Combining this with Propositions \ref{lemma1} and \ref{lemma2} and 
\eqref{close} implies that there exist positive 
constants $b$ and $c$ such that
\[
\P( \wdist \geq c n^{15/8} - 2bn^{7/4} - 2 \pi n) \to 1
\]
as $n \to \infty$. For sufficiently large $n$ the 
quantity 
$c n^{15/8} - 2cn^{7/4} - 2 \pi n$ is larger than 
${c \over 2} n^{15/8}$. 
Incorporating an extra factor of $\half$ into the constant $c$
yields Lemma \ref{lemmaa}.

So it remains only to verify that 
$\e(Z) \geq c n^{15/8}$, for a universal constant $c$. 
Recall that $\tau_k(s)$ 
denotes the $k$th time that the hole is to the right of tile $s$,
and
$(X_{\tau_1(s)}(s), X_{\tau_2(s)}(s), \dots)$ is a simple symmetric random walk on 
$\Z_n$ with a holding probability.
Since the second eigenvalue for this walk $\lambda$ 
satisfies $\lambda > \cos {2 \pi \over n}$, it follows that 
for all $t$ we have
$\e \left( f( X_{\tau_t(s)}(s) \given X_{\tau_1(s)}(s) \right) \geq 
f(X_{\tau_1(s)}(s)) \lambda^{t-1}$, and since 
$f(X_{\tau_1(s)}(s)) \geq f(X_0(s)) - {2 \pi \over n}$ 
it follows that 
\[
\e \left( f( X_{\tau_t(s)}(s)) \right) \geq 
\left( f(X_0(s)) - {2 \pi \over n} \right) \lambda^{t-1}.
\]
Substituting $t = \th$ and summing over $s \in S$ gives
\[
\e Z \geq \left [ 
\sum_{s \in S} 
f( X_0(s)) - {2 \pi |S| \over n}
\right] 
\lambda^{\th-1}.
\]
The expression in square brackets can be bounded below 
by $cn^2$ for a universal constant $c$, since for 
every $s \in S$ we have $f( X_0(s)) \geq \half$. 
Furthermore, since $\th - 1 \leq \epsilon n^2 \log n$ 
by (\ref{epsdef}) and 
$\lambda^{n^2} \geq e^{-\mu}$ by (\ref{mudef}), it follows
that
\begin{eqnarray*}
\e Z &\geq& c n^2 \exp(- \mu \epsilon \log n) \\
&=& c n^{15/8}.
\end{eqnarray*}
(Recall that $\mu \epsilon = 1/8$.)
This verifies the claim
and hence proves the lemma. 
\end{proofof}
\subsection{Proofs of Propositions \ref{lemma1} and \ref{lemma2}}
\label{l12proofs}
It remains to prove propositions
\ref{lemma1} and \ref{lemma2}, which were used in the 
proof of Lemma \ref{lemmaa}.  This is done
is subsections \ref{l1proof} and \ref{l2proof}, respectively.
\subsubsection{
Proof of Proposition \ref{lemma1}}  
\label{l1proof}
Recall that $\nt_t(s)$ denotes the number of times the hole 
has been to the immediate right of tile $s$, up to time $t$.
The main step in the proof of Proposition
\ref{lemma1} is to show that $\nt_t$ 
is well approximated by $t(n^2 - 1)^{-1}$. 
We accomplish this using the second moment method. 

In order to bound the mean and variance of $\e( \nt_t(s))$, 
we use the fact that
the position of the hole relative to tile $s$ 
(that is, the position of the hole minus the position of tile $s$)
behaves   
like a random walk on a certain graph. 
Let $\gt_n$ be the graph obtained from $G_n$ by deleting the 
origin and adding an edge  from $(-1, 0)$ to $(1, 0)$ and 
an edge from $(0, 1)$ to $(0, -1)$. (Figure \ref{graph}
shows $\gt_n$ when $n=5$.) 
Note that 
if $H_t$ denotes the the position of the hole
at time $t$ in the Loyd chain, then 
$H_t - L_t(s)$
is the same random process as
a random walk on $\gt_n$. The times $\taut_k(s)$ 
coincide with the times when the random walk on $\gt_n$ is at the 
vertex $(1,0)$.
   In Lemmas \ref{elemma} and \ref{varlemma} below,
we use the connection to the random walk on $\gt_n$ to bound the 
mean and variance of $\nt_t(s)$. 
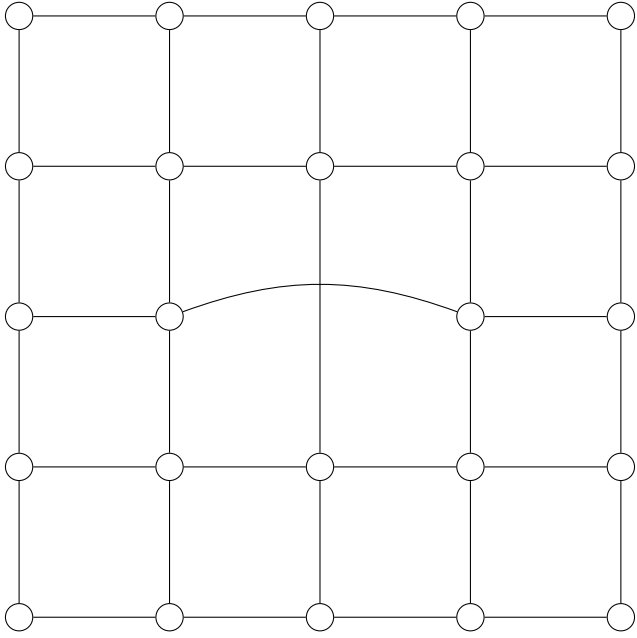
\begin{figure}[t]
\begin{tikzpicture}[xscale = 2, yscale = 2] 
\node at (-2, 2) [circle, draw] (1) {};
\node at (-1, 2) [circle, draw] (2) {};
\node at (0, 2) [circle, draw] (3) {};
\node at (1, 2) [circle, draw] (4) {};
\node at (2, 2) [circle, draw] (5) {};
\draw (1) -- (2) -- (3) -- (4) -- (5);

\node at (-2, 1) [circle, draw] (a1) {};
\node at (-1, 1) [circle, draw] (a2) {};
\node at (0, 1) [circle, draw] (a3) {};
\node at (1, 1) [circle, draw] (a4) {};
\node at (2, 1) [circle, draw] (a5) {};
\draw (a1) -- (a2) -- (a3) -- (a4) -- (a5);

\node at (-2, 0) [circle, draw] (b1) {};
\node at (-1, 0) [circle, draw] (b2) {};
%\node at (0, 0) [circle, draw] (b3) {};
\node at (1, 0) [circle, draw] (b4) {};
\node at (2, 0) [circle, draw] (b5) {};
\draw (b1) -- (b2);
\draw (b2) to [bend left = 20] (b4);
\draw (b4) -- (b5);

\node at (-2, -1) [circle, draw] (c1) {};
\node at (-1, -1) [circle, draw] (c2) {};
\node at (0, -1) [circle, draw] (c3) {};
\node at (1, -1) [circle, draw] (c4) {};
\node at (2, -1) [circle, draw] (c5) {};
\draw (c1) -- (c2) -- (c3) -- (c4) -- (c5);

\node at (-2, -2) [circle, draw] (d1) {};
\node at (-1, -2) [circle, draw] (d2) {};
\node at (0, -2) [circle, draw] (d3) {};
\node at (1, -2) [circle, draw] (d4) {};
\node at (2, -2) [circle, draw] (d5) {};
\draw (d1) -- (d2) -- (d3) -- (d4) -- (d5);

\draw (1)--(a1)--(b1)--(c1)--(d1);
\draw (2)--(a2)--(b2)--(c2)--(d2);
\draw (3)--(a3);
\draw  (a3) to (c3);
\draw (c3)--(d3);
\draw (4)--(a4)--(b4)--(c4)--(d4);
\draw (5)--(a5)--(b5)--(c5)--(d5);
\end{tikzpicture}
\caption{Graph $\gt_n$. (Edges connecting top row to bottom row
and edges connecting leftmost row to rightmost row are not shown.)}\label{graph}
\end{figure}
\begin{lemma}
\label{elemma}
There is a universal constant $A$ such that for 
any tile $s$ and time $t$ we have
\begin{equation}
\label{eclaim}
\l | \e ( \nt_t(s))  - t(n^2-1)^{-1} \r | \leq A \log t.
\end{equation}
\end{lemma} 
\begin{proof}
Let $\{p(x,y)\}$ be transition probabilities for random walk on $\gt_n$.
Lemma \ref{hkbound} in Appendix A
states that 
there is a universal constant $A > 0$ such that 
\begin{equation}
\label{xybound}
\l| p^t(x,y) - \pi(y) \r| \leq {A \over t}, 
\end{equation}
for all $t \geq 1$, where $\pi(y)$ is the stationary probability 
$(n^2-1)^{-1}$.   
Since the hole is not initially to the right of tile $s$, 
using (\ref{xybound}) with $x = H_t - L_0(s)$
and $y = (1,0)$ gives
\begin{eqnarray}
\l | \e( \nt_{t}(s)) - t \pi(y) \r | &\leq& 
\sum_{k = 1}^t  {A \over k} \\ 
%\nonumber \\
\label{qqq}
        &\leq& A \log t .
\end{eqnarray} 
%When $T \leq \epsilon n^2(n^2-1) \log n$, the quantity (\ref{qqq}) is 
%at most
%\[
%C \epsilon n^2 \log n,
%\]
%for a constant $C$. 
\end{proof} 
Next we bound the variance of $\nt_t(s)$. 
%We claim that for all $s$, we have  $\var( \nt_T(s)) \leq C n^2 \log^2 n$,
%for a universal constant $C$.
%
%This is a consequence of the following lemma.
%
\begin{lemma}
\label{varlemma}
There is a universal constant $C$ such that for any tile $s$ we have
\[
\var( \nt_t(s)) \leq C n^{-2} t\log t,
\]
whenever $n^2 \log n \leq t \leq n^5$. 
\end{lemma}
\begin{proof}
Fix a tile $s$ and 
for $i$ with $1 \leq i \leq t$, let $I_i$ be the 
indicator of the event that the hole is to the right of tile $s$ at
time $i$. Then $\nt_t(s) = \sum_{i=1}^t I_i(s)$, and hence
\begin{eqnarray}
\label{sst}
\var( \nt_t(s) ) &=& 
 \sum_{i=1}^t \var(I_i) + 2  \sum_{1 \leq i < j \leq t} \cov(I_i I_j).
\end{eqnarray}
The first term is at most $\e(\nt_t(s))$ (since for each $i$ 
we have $\var(I_i) \leq \e(I_i^2) \leq \e(I_i)$) and
recall that Lemma \ref{elemma} implies that $\e( \nt_t(s))$ is
at most $t (n^2 -1)^{-1} + A \log t$. 
To bound the second 
term in (\ref{sst}), note that for each $i$ and $j$ with $i < j$ we have
\begin{eqnarray*}
\cov(I_i, I_j) &=& \e( I_i I_j) - \e(I_i) \e(I_j) \\
&\leq& 
\Bigl( \pi(y) + {A \over i} \Bigr)
\Bigl( \pi(y) + {A \over j-i}  \Bigr) -
\Bigl( \pi(y) - {A \over i} \Bigr)
\Bigl( \pi(y) - {A \over j}  \Bigr),
\end{eqnarray*}
where in the last line we used 
Lemma 
\ref{hkbound} to bound each 
expectation. Expanding each product and then collecting terms gives
\[
\Bigl[ {2A \over i} + {A \over j} + {A \over j-i}
\Bigr] \pi(y) + A^2 \Bigl[ {1 \over i(j-i)} - {1 \over ij} \Bigr].
\]
If we sum this over $j$ with $i < j \leq t$, then the result is at most
\begin{equation}
\nonumber
%\label{vbound}
\Bigl[ {2At \over i} + 2A \log t \Bigr] \pi(y)
+ {A^2 \over i}  \log t.
\end{equation}
If we 
sum this over $i$ with $1 \leq i \leq t$, then the result is at most
\begin{equation}
\label{qq}
\Bigl[ 2A t \log t + A t \log t \Bigr] \pi(y) +
2A^2 \log^2 t, 
\end{equation}
which is of the form $\bigo(n^{-2} t \log t) + \bigo( \log^2 n)$.
(Note that 
since $t \leq n^5$, we have 
$\log^2 t = \bigo( \log^2 n)$.) The result follows if we note that
$\log n = n^{-2} ( n^2 \log n)$, which is at most $n^{-2} t$ 
whenever $n^2 \log n \leq t$. 
\end{proof}

We will need one more lemma before proving
Proposition \ref{lemma1}, but 
first we recall Hoeffding's bounds for sums of independent random 
variables. 
\begin{theorem} (\cite{H})
Let $Y_1, Y_2, \dots$ be i.i.d.~random variables 
and suppose that $\e(Y_1) = 0$ and $|Y_1| \leq 1$. 
Define $S_m = \sum_{i=1}^m Y_i$. 
Then for all positive integers $s$ and $t$ we have
\[
\p( |S_t - S_s| \geq \alpha) \leq 2 e^{- \alpha^2/2 |t-s|}.
\]
\end{theorem}
%The next lemma follows from the
%fact that a symmetric random walk 
%is unlikely to deviate by much
%more than order $\sqrt{t}$ from its starting point
%after $t$ steps.
\begin{lemma}
\label{devlemma}
Let $Y_1, Y_2, \dots$ be i.i.d.~random variables 
and suppose that $\e(Y_1) = 0$ and $|Y_1| \leq 1$. 
Define $S_m = \sum_{i=1}^m Y_i$. 
Fix constants $C > 0$ and  $\beta$ with $\half < \beta < \sfrac{3}{4}$. 
For positive integers $n$ define 
$M_n = \max { |S_t - S_s| \over |t - s|^\beta}$, where the maximum is 
over $s$ and $t$ such that  
\begin{equation}
\label{conds}
0 \leq s \leq Cn^4 \log n; \;\;\;\;\;\;\;\;\;
0 \leq t \leq Cn^4 \log n; \;\;\;\;\;\;\;\;\;
| s -t | \geq \sqrt{n}.
\end{equation}
Then for every $p > 1$ there is a 
constant $C_p$, which depends only on $p$, such that 
\[
\e(M_n^p) \leq C_p.
\]
\end{lemma}
\begin{proof}
Since each $M_n$ is bounded it is enough to show that 
$\limsup_{n \to \infty} \e(M_n^p) < \infty$. 
If $|s - t| > \sqrt{n}$ then 
applying Heoffding's bounds with $\alpha = c |t-s|^\beta$
gives
\begin{equation}
\label{probbound}
\p \Bigl( { |S_t - S_s | \over |t - s|^\beta } > c \Bigr)
\leq 2 \exp\left(- {c^2 \over 2} n^{\beta - \half} \right).
\end{equation}
Define $p_n(c) := \p( M_n > c)$.
There are at most $C^2 n^{10}$ pairs $(s,t)$ that 
satisfy the conditions in (\ref{conds}). Thus if $n$ 
is large enough so that for all $c \geq 1$ we have
\[
2 C^2 n^{10}  \exp\left( - {c^2 \over 2}  n^{\beta - \half} \right)
\leq e^{-c^2},
\]
a union bound implies that for all $c \geq 1$ 
we have $p_n(c) \leq e^{-c^2}$ and hence
\begin{eqnarray*}
\e(M_n^p) &=& \int_0^\infty \P(M_n^p > t) \; dt \\
&\leq& \int_0^\infty p_n( t^{1/p}) \; dt \\
&<& \infty.
\end{eqnarray*}
\end{proof}
Now that we have Lemmas \ref{elemma}, \ref{varlemma}
and \ref{devlemma},
we are ready to prove Proposition \ref{lemma1}
\begin{proofof}{Proof of  Proposition \ref{lemma1}}
Since $T$ is $\bigo(n^4 \log n)$, 
applying Lemma \ref{varlemma} with 
$t = T$
implies that when $n$ is sufficiently large, we have
$\var( \nt_T(s)) \leq C n^2 \log^2 n$. 
It follows that 
\begin{eqnarray*}
\e \l | \nt_T(s) - \th \r|  
&\leq& 
\e \Bigl( \l | \nt_T(s) - \e \nt_T(s) \r | \Bigr) 
+ 
\l | \e \nt_T(s) - \th \r|  \\
&\leq& \sqrt{C} n \log n + A \log n,
\end{eqnarray*}
where in the second line we have used 
the inequality $\e | X - \e X  | \leq \sd(X)$, 
valid for all random variables $X$, to bound the first term
and Lemma \ref{elemma} to bound the second term.  
It follows that
\begin{equation}
\label{important}
\e \l | \nt_T(s) - \th \r|   \leq B n \log n,
\end{equation}
for a universal constant $B$.  

Let $S_k = X_{\tau_k(s)}(s) - 
X_{\tau_1(s)}(s)$, that is, the change of the $s$ random walk 
after $k-1$ steps. Note that 
we can write $S_k$ as $Y_1 + Y_2 + \cdots Y_k$, where
the $Y_i$ are i.i.d.~$\pm 1$ random variables. 
Fix $\beta \in (\half, \sfrac{3}{4})$.
Since $| S_{\nt_T(s)} - S_{\th} | \leq | \nt_T(s) - \th|$,
and since $\nt_T(s)$ and $\th$ can both be bounded 
above by $C n^4 \log n$ for a universal constant $C$, it
follows that 
if $M_n$ is defined as in the statement of Lemma \ref{devlemma},
then
\begin{equation}
\label{trunc}
| S_{\nt_T(s)} - S_{\th} | \leq M_n |\nt_T - \th|^\beta + \sqrt{n} 
\one\Bigl(| \nt_T(s) - \th| \leq \sqrt{n}\Bigr).
\end{equation}
Let $C_p$ be the constant from Lemma \ref{devlemma}.
Applying H\"older's inequality with $p = {1 \over 1-\beta}$ 
and $q = {1 \over \beta}$ gives
\begin{eqnarray*}
\e\left( M_n | \nt_T(s) - \th|^\beta \right) &\leq& 
\e( M_n^p)^{1/p} \left( \e|  \nt_T(s) - \th| \right)^\beta  \\
&\leq& C_p^{1/p} \left[ B n \log n \right]^\beta, 
\end{eqnarray*}
where
in the last line we have used Lemma \ref{devlemma}
to bound 
$\e( M_n^p)^{1/p}$ 
and \eqref{important} to bound
$\e|  \nt_T(s) - \th|$. 

Taking expectations in (\ref{trunc}) shows that 
there is a constant $B>0$ 
such that
\begin{eqnarray*}
\e| S_{\nt_T(s)} - S_{\th} | &\leq& 
C_p^{1/p} \left[ B n \log n \right]^\beta + \sqrt{n}. 
\end{eqnarray*}
Hence there is a 
$\gamma \in (\beta, \sfrac{3}{4})$ 
such that 
\begin{eqnarray}
\label{thirtyprimeprime}
\e| S_{\nt_T(s)} - S_{\th} | 
&\leq& B n^\gamma.
\end{eqnarray}
Since 
$S_{\nt_T(s)}(s) - S_\th(s) = 
X_{\tau_{\nt_T}(s)}(s) -  X_{\tau_{\th}(s)}(s)$
from the definition of $S_k$,  
combining (\ref{thirtyprimeprime})
with  (\ref{close}) gives
\begin{eqnarray}
\e \Bigl( 
\left|
f( X_{\tau_{\nt_T}(s)}(s) ) - f(X_{\tau_{\th}(s)}(s))
\right|
\Bigr) &\leq& {2 \pi B \over n} n^\gamma \\
\label{fred}
&=& B' n^{\gamma - 1}
\end{eqnarray}
for a constant $B'$. 
Summing
(\ref{fred}) over 
$s \in S$ gives
\[
\e \Bigl( 
\left|
\sum_{s \in S} 
f( X_{\tau_{\nt_T}(s)}(s)) - \sum_{s \in S} f(X_\fin(s)) \right|
\Bigr)
\leq B' n^{\gamma + 1}.
\]
Combining this with Markov's inequality yields the proposition,
since $\gamma < \sfrac{3}{4}$.
\end{proofof}

\subsubsection{Proof of Proposition \ref{lemma2}}
\label{l2proof}
We  prove Proposition \ref{lemma2} using 
the method of bounded differences.
The main step is to show that each step of the Loyd process
has a small effect on the conditional expectation of $Z$,
which we prove via Lemma \ref{bd} below.

 Define $X_\fin(s) = X_{\tau_\th(s)}(s)$
and define $\ff(s) := f( X_\fin(s))$, so that
we can write $Z$ as
\[
Z = \sum_{s \in S} \ff(s).
\]
Let $\h_t = (L_0, L_1, \dots, L_t)$ be the history of the Loyd
process 
up to time $t$. We call the Markov chain $(\h_t: t \geq 0)$ the 
{\it history process.} If $H = (L_0, \dots, L_k)$ 
is a state of the history process, we write 
$L(H)$ for the Loyd configuration $L_k$. 

Let $\h \to \hh$ be a possible transition of the history 
process. We aim to compare the distribution of 
$Z$ when the history process starts 
at $\h$ versus when it starts from $\hh$. 
We shall refer to the history process started from 
$\h$ (respectively, $\hh$) as the 
{\it primary} (respectively, {\it secondary}) 
history process. \\
\\    
{\bf Convention.} 
If a random variable $W$ is defined in terms of the \aux process,
we write ${\widehat W}$ for the corresponding random variable 
defined in terms of the secondary process, and similarly for events. 

\begin{lemma}
\label{bd}
We have
\[
| \e ( Z ) - \e(\zhat) | \leq {D \log n \over n},
\]
for a universal constant $D$.
\end{lemma}
\begin{proof}
Our main tool is coupling. 
Note that to 
demonstrate a coupling 
of the primary and secondary history processes, it is 
sufficient to demonstrate a coupling of the Loyd process 
started from $L := L(\h)$ and the Loyd process started from 
$\lhat := L(\hh)$. We call these processes the primary and secondary 
Loyd processes, respectively.

We start by bounding $|\e( \ff(s)) - \e( \fhf(s))|$ 
for the case
when $s$ is the tile swapped with the hole in the transition from 
$L$ to $\lhat$. 
We can couple the secondary Loyd process with the \aux Loyd process
so that the way that the hole moves after the first time 
it is to the right of tile $s$ is the same in both processes. Since with this 
coupling we have $| X_\fin(s) - \xhat_\fin (s) | \leq 1$, equation (\ref{mvt})
implies that
\begin{equation}
\label{swapped}
| \e( \ff(s)) - \e( \fhf(s)) | \leq {2 \pi \over n}.
\end{equation}

Let $S'$ be the set of tiles in $S$
that are not
swapped with the hole in the transition from 
$L$ to $\lhat$. We now consider the tiles in $S'$. It will be convenient to 
group the tiles in columns (i.e., group them according to their $x$-coordinates)
and then consider the columns one at a time.

Let $H_t$ be the location of the hole at time $t$
in the primary Loyd process, 
and suppose that $H_0 = (h_x, h_y)$. Let $C$ be a {column} 
in $V_n$, that is, a set of the form
$\{(j, k): k \in \Z_n \}$ for some $j \in \Z_n$, and suppose that 
$|h_x - j| = d$ (that is, 
the hole is initially a distance $d$ from $C$),
where $d \in \{0,1,2, \dots\}$. 
We claim that there is a universal constant $D$ such that 
\begin{equation}
\label{claimbd}
 \l| \sum_{s \in S' \cap C}
\e( \ff(s) - \fhf(s)) \r| 
\leq {D \over n(d+1)}.
\end{equation} 
Summing this over columns $C$ and combining this 
with (\ref{swapped}) proves the Lemma.

We now prove the claim. 
We verify (\ref{claimbd}) by constructing a coupling 
of the \aux Loyd process and the secondary  Loyd process. 
The coupling is designed so that if the hole is initially far 
away from column $C$, then $H_t$ is likely to couple with 
$\hhat_t$ before it gets close to column $C$. 

Let $C_L$ and $C_R$  be the columns to the immediate left and 
right, respectively, of $C$.
We now give a rough description of the coupling.
The nature of the coupling will depend on whether the hole moves
horizontally or vertically in the transition from $L$ to $\lhat$. 
If the hole moves horizontally (respectively, vertically),
then the trajectory of $H_t$ %(before the coupling time)
is the reflection of the trajectory of $\hhat_t$ 
about a vertical (respectively, horizontal) axis,
up until the time when either the holes have coupled or 
one of them has reached column $C, C_R$ or $C_L$. 
We now give a more formal description in the case 
where $\hhat_0 = (h_x + 1, h_y)$. (The other cases are similar.
In the case where the hole moves vertically in the transition 
from $L$ to  $\lhat$, the coupling is the same, except that
the roles of vertical and horizontal moves are reversed.)
\\
\\
{\bf The coupling in the case where $\hhat_0 = (h_x + 1, h_y)$} 
\begin{enumerate}
\item If $H_t = \hhat_t$, then we couple so that
$H_{t+1} = \hhat_{t+1}$;

\item else, if either $H_t$ or $\hhat_t$ is in column $C$, $C_R$ or 
$C_L$, then the holes move independently;

\item 
else, if $H_t$ is to the immediate left of $\hhat_t$,
we use the following rule. \\
~\\

\begin{tabular}{lll}
\aux & secondary & probability \\
\hline
$\leftarrow$   &   $\rightarrow$   & $1/8$ \\
$\rightarrow$   &   do nothing  & $1/8$ \\
$\uparrow$   &   $\uparrow$   & $1/8$ \\
$\downarrow$   &   $\downarrow$   & $1/8$ \\
do nothing  &   $\leftarrow$   & $1/8$ \\
do nothing   &  do nothing  & $3/8$ \\
\end{tabular}

\item 
else, 
we use the following rule. \\
~\\
\begin{tabular}{lll}
\aux & secondary & probability \\
\hline
$\leftarrow$   &   $\rightarrow$   & $1/8$ \\
$\rightarrow$   &  $\leftarrow$   & $1/8$ \\ 
$\uparrow$   &   $\uparrow$   & $1/8$ \\
$\downarrow$   &   $\downarrow$   & $1/8$ \\
do nothing   &  do nothing  & $1/2$ \\
\end{tabular}

\end{enumerate}
Note that 
if the $x$-coordinate of $H_t$ takes the value $h_x +1$ before
either $H_t$ or $\hhat_t$ hits $C, C_R$ or $C_L$  then the holes couple
before either of them affects tile $s$. 
%or one of the holes hits one of the columns $C$, $C_L$ or $C_R$.

Let $\ttil$ be the first time either $H_t$ or $\hhat_t$ hits 
columns $C, C_L$ or $C_R$. Let $E$ be the event that the 
holes have not coupled before time $\ttil$. 
We claim that 
\begin{equation}
\label{coupleprob}
\P(E) \leq {C \over d + 1},
\end{equation} 
for a universal constant $C$. 
(Recall that $d$ is the initial distance between the hole and
column $C$.) It is enough to verify (\ref{coupleprob}) in the 
folowing two cases, since we can always reduce to one
of these cases by interchanging the roles of $H_t$ and $\hhat_t$
if necessary:
\begin{enumerate}
\item $\hhat_0$ is to the immediate right of $H_0$.

\item $\hhat_0$ is immediately below $H_0$.

\end{enumerate}
In the first case, (\ref{coupleprob})
follows from part (i) of Lemma \ref{randomwalk}
in Appendix B,
since the event $E$ occurs only 
if time $\ttil$ occurs before the $x$-coordinate of
$H_t$ takes the value $h_x + 1$. 
In the second case, (\ref{coupleprob}) 
follows from part (ii) of Lemma \ref{randomwalk},
since in this case the event $E$ occurs only 
if time $\ttil$ occurs before the $y$-coordinate of
$H_t$ takes the value $h_y - 1$. 

 Let $T_C$ be the first time that the hole is in column $C$. 
For tiles $s \in S$ that are initially in column $C$,   
let $T_R(s)$ (respectively, 
$T_L(s)$) 
be the first time that the hole is to the immediate 
right (respectively, left) of tile  
$s$.
Let $R_s$ be the event that 
$T_R(s) = \min( T_R(s), T_L(s), T_C)$ and
%with a similar definition for $\ahat$. 
let
$L_s$ be the event that 
$T_L(s) = \min( T_R(s), T_L(s), T_C)$.
%with a similar definition for $\bhat$. 
Define
\[
\zplus = \e \left( \ff(s) \given T_R < T_L \right),
\;\;\;\;\;\;\;\;\;
\zminus = \e \left( \ff(s) \given T_L < T_R \right).
\]
Note that (\ref{mvt}) implies that
\begin{equation}
\label{close2} 
|\zplus - \zminus | \leq {2 \pi \over n}.
\end{equation} 
  
We say that the hole is {\it beside} a tile if it is to its 
immediate right
or immediate left. 
Note that 
if the hole starts in the same column as tile $s$, then the next time the
hole is beside tile $s$ 
it is equally likely to be to its right as 
to its left. It follows that
\[
\e(\ff(s)) = \P(R_s)\zplus + \P(L_s) \zminus + 
[1 - \P(R_s) - \P(L_s)] (\half \zplus + \half \zminus).
\]
Rearranging terms gives
\begin{equation}
\label{rear1}
\e(\ff(s)) = 
\half \Bigl[ (\zplus + \zminus)
+ \P(R_s) (\zplus - \zminus) + \P(L_s) (\zminus - \zplus) \Bigr].
\end{equation}
Similarly, we also have
\begin{equation}
\label{rear2}
\e(\fhf(s)) =    
\half \Bigl[ (\zplus + \zminus)
+ \P(\ahat) (\zplus - \zminus) + \P(\bhat) (\zminus - \zplus) \Bigr].
\end{equation}
Replacing each probability in 
(\ref{rear1})
and (\ref{rear2})
with the expectation of an appropriate indicator 
random variable, and then subtracting 
(\ref{rear2})
from (\ref{rear1}), gives
\begin{equation}
\label{diff}
\e(\ff(s)) - \e(\fhf(s)) = 
\half \e( \one_{R_s} - \one_{\ahat} ) \Delta -
\half \e( \one_{L_s} - \one_{\bhat} ) \Delta, 
\end{equation}
where $\Delta := \zplus - \zminus$. 
Hence 
\[
\l | \e(\ff(s)) - \e(\fhf(s)) \r |  \leq 
|\Delta| \max( 
\e( \one_{R_s} - \one_{\ahat}), \e( \one_{L_s} - \one_{\bhat})). 
\] 
%Recall that $|\Delta| \leq {2 \pi \over n}$ by (\ref{close}).
Note that 
$\one_{R_s} - \one_{\ahat}$ and 
$\one_{L_s} - \one_{\bhat}$ are both $0$
on the event that the holes couple before either one hits $C_R$ or $C_L$. 
It follows that 
\begin{equation}
\label{diff2}
\l | \e(\ff(s)) - \e(\fhf(s)) \r |  \leq
|\Delta| \cdot  \e ( Y(s) +  \yhat(s) ),
\end{equation}
where $Y(s)$ is the indicator of the event that the hole is beside 
tile $s$ before time $T_C$.
%, with a similar definition for 
%$\yhat(s)$. 
Let $Y = \sum_{s \in C} Y(s)$ be the total number of positions
in column $C_L$ and $C_R$ visited before time $T_C$. 
Summing over $s \in C$ gives
\begin{equation}
\label{sumbound}
\sum_{s \in C} 
\l | \e(\ff(s)) - \e(\fhf(s)) \r |  \leq
|\Delta| \cdot \e (Y + \yhat)
\end{equation}
Note that  $Y$ and $\yhat$ are both $0$ 
unless the event $E$ occurs and recall that  (\ref{coupleprob}) gives 
$\P(E) \leq {C \over d+1}$. Furthermore, 
the condional distribution 
of both $Y$ and $\yhat$ given $E$ is geometric($\quarter$), 
since each time the hole is in column $C_R$ or 
$C_L$, it moves to column $C$ in the next step with probability $\quarter$. 
It follows that 
\begin{eqnarray}
\e( Y + \yhat) &\leq& {C \over d + 1} \e( Y + \yhat \given E) \\
\label{esmall}
&=& {8C \over d + 1}.
\end{eqnarray} 
Finally, recall that $\Delta = \zplus - \zminus$
and hence $|\Delta| \leq
{2 \pi \over n}$ 
by (\ref{close2}). Combining this with 
(\ref{sumbound}) 
and (\ref{esmall}) vertifies 
(\ref{claimbd}), which proves the lemma.
\end{proof}

Now that we know there are bounded
differences, we are ready to prove Proposition \ref{lemma2}:

\begin{proofof}{Proof of Proposition \ref{lemma2}}
We need to show that for any $b > 0$ we have
\[
\P( |Z - \e(Z)| > b n^{7/4}) \to 0
\]
as $n \to \infty$, where $Z = \sum_{s \in S} f( X_\fin(s))$. 

Recall that $\taut_k(s)$ is the $k$th time at which the 
hole is to the immediate right of tile $s$. 
Define $\tau = \max_{s \in S} \taut_\th(s)$.
Let $\f_t = \sigma(L_1, \dots, L_t)$
and consider the
Doob martingale
\[
M_t := \e ( Z \given \f_t).
\]
The idea of the proof will be to evaluate the 
martingale at a suitably chosen time $K$. The value of $K$ will be 
chosen  
to be large enough so that $\tau \leq K$ with high probability,
but small enough so that the Azuma-Hoeffding 
inequality will give a good 
large deviation bound for $M_K$. To these ends, we choose $K = n^5$. 
Note that $Z$ is determined by time $\tau$. 
Hence
$M_K = Z$ unless $\tau > K$.
Furthermore, we have $\e(M_{K}) = \e(Z)$. It  
follows that
\begin{equation}
\label{twothings}
\P( |Z - \e(Z)| > b n^{7/4}) \leq \P( |M_{K} - \e(M_{K})|
> b n^{7/4}) + \P( \tau > K).
\end{equation}
We now bound each term on the righthand side of (\ref{twothings}).
We start with the first term. Lemma \ref{bd} 
implies that
\[
| M_t - M_{t-1} | \leq {D \log n \over n},
\]
for $t$ with $1 \leq t \leq K$. Thus the Azuma-Hoeffding bound gives
\begin{eqnarray}
\label{sum}
\P( | M_K - \e(M_K) | \geq x ) 
&\leq& 2 \exp \left( {-x^2 \over 2\sum_{i=1}^K C^2 } \right),
\end{eqnarray}
where $C = {D \log n \over n}$. Substituting $x = b n^{7/4}$
and $K = n^5$ into (\ref{sum}) gives
\begin{eqnarray}
%\label{sum}
\P( | M_{K} - \e(M_{K}) | \geq b n^{7/4}  ) 
&\leq& 2 \exp \left( {-b^2 n^{7/2} \over 2 n^3 B^2 \log^2 n} \right) \\
&=& 2 \exp \left( {-b^2 n^{1/2} \over 2  B^2 \log^2 n} \right),
\end{eqnarray}
which converges to $0$ as $n \to \infty$.

Next, we bound $\P( \tau > K)$. 
Note that $\taut_\th(s) \leq K$ whenever $\nt_{K}(s) \geq \th$.
Furthermore, since $K = n^5$, 
Lemmas \ref{elemma} and \ref{varlemma} imply 
that for sufficiently large $n$ 
we have
\[
\e( \nt_{K}(s)) \geq n^3 - \bigo(\log n); \;\;\;\;\;\;\;\;\;
\var (\nt_{K}(s)) = \bigo(n^3 \log n).
\] 
Note also that $\th$ is $\littleo(n^3)$. 
Thus Chebyshev's inequality implies that
$\P( \nt_K(s) < \th)$ is $\bigo \left( {\log n \over n^3} \right)$,
and hence 
$\P( \tau_\th > K)$ is $\bigo \left( {\log n \over n^3} \right)$.
Thus 
a union bound implies that
$\P( \tau > K)$ is $\bigo \left( {\log n \over n} \right)$,
and hence converges to $0$ as $n \to \infty$. This completes the 
proof. 
\end{proofof}

\section{Appendix A: Probability bounds for random 
walk on $\gt_n$}
In this section we 
derive bounds on transition probabilities for random 
walk on $\gt_n$. 
First, we give some definitions and 
extract some necessary results from \cite{hk}.

Let $\{q(x,y)\}$ be transition probabilities for a Markov chain
on a finite state space $V$ with stationary distribution $\pi$.
For $S \subset V$, define the ``boundary size'' 
$|dS| = \sum_{x \in S, y \in S^c} \pi(x) q(x, y)$. 
Following \cite{JS}, we call 
$\phi_S:=\frac{|\d S|}{\pi(S)}$ the {\em conductance\/} of $S$.
Write $\pi_*:= \min_{x \in V} \pi(x)$ and
 %define $\phi:[\pim,1) \to \R^+$ by
%\be \lab{defphi}
%\phi(r) = \inf 
%\left\{ { |\d S| \over \pi(S) \wedge \pi(S^c)}: \pi(S) \leq r
%\right\}
define $\phi(r)$ for $r \in [\pim,1/2]$ by
\be \lab{defphi}
\phi(r) = \inf 
\left\{  \phi_S : \pi(S) \leq r \right\}
\, .
\ee
%where $\alpha \wedge \beta := \min\{\alpha,\beta\}$.
%Note that $\phi$ is monotone decreasing and is 
%constant on the interval $[1/2,1)$, so that
For $r>1/2$, let $\phi(r)=\phi(1/2)$.
We call $\phi$ the {\it isoperimetric profile}.  
We recall the following theorem from \cite{hk}.

\begin{theorem} 
\label{hk2}%(\cite{hk})
Suppose that $q(x,x) \ge \half$ for all $x \in V$.
If
\be \lab{inteq2}
t \ge 1 +
 \int_{ \pi_* }^{4/\epsilon}
\frac{4 du}{u \phi^2(u)} \,,
%+ \frac{4}{\phi^{2}_*} \log \frac{8}{\epsilon} \, ,
\ee
then
\be \lab{unif}
\l| \frac{q^t(x,y) - \pi(y)}{\pi(y)} \r| \leq \epsilon.
\ee
\end{theorem}
%
%In this section we 
%use the fact that the isoperimetric profile for 
%random walk on $G$ is similar to that for 
%random walk in the two-dimensional torus
%$\gt := \z_n^2$ to derive heat kernel bounds
%for random walk on $G$. 
%
\begin{lemma} 
\label{hkbound}
Let $\{p(x,y\}$ be transition probabilities 
for the lazy random walk on $\gt_n$ and let $\pi$ be the stationary 
distribution. 
There is a universal constant $A > 0$ such that 
\begin{equation}
\l| p^t(x,y) - \pi(y) \r| \leq {A \over t},  %\\
%&\leq& {2 \over n^2} + {B \over n}.
\end{equation}
for all $t \geq 1$.  
\end{lemma}
\begin{proof}
Recall that $G_n$ denotes the $n \times n$ torus $\Z_n^2$. 
We write $\phi$ (respectively, $\phit$) 
for the conductance profile for the lazy random 
walk on $G_n$ (respectively, $\gt_n$). It is well known that 
$\phi$ satisfies
\begin{equation}
\label{isot}
\phi(u) \geq {C \over n \sqrt{u}},
\end{equation} 
for a universal constant $C > 0$.

Let $\vt_n$ be the vertex set of $\gt_n$. 
Since for $S \subset \vt_n$, the boundary size and stationary probability 
of $S$, with respect to random walk on $\gt_n$, 
are within constant factors of the corresponding quantities 
with respect to 
random walk on $G_n$, it follows that the conductance 
profile $\phit$ for random walk on $\gt_n$ satisfies the 
similar inequality
\begin{equation}
\label{iso}
\phit(u) \geq {\ct \over n \sqrt{u}},
\end{equation} 
for a universal constant $\ct > 0$.

Fix $0 < \alpha < 1$. 
Using Theorem \ref{hk2} with $\epsilon = \alpha/\pi(y)$ gives 
\begin{equation}
\lab{hkb}
\l| p^t(x, y) - \pi(y) \r| \leq \alpha
\end{equation}
whenever
\begin{equation}
\label{inti}
t \geq 1 + \int_{\pi_*}^{4\pi(y)/\alpha} {4 du \over u \phit^2(u)} .
\end{equation}
Equation (\ref{iso}) implies that
the righthand side of  (\ref{inti}) is at most
\begin{eqnarray*}
1 + \int_{\pi_*}^{4\pi(y)/\alpha} 4C^{-2} n^2  \, du &\leq&
1 + {16 \pi(y) n^2 \over C^2 \alpha } \\
 &\leq& {A  \over  \alpha},
\end{eqnarray*}
for a universal constant $A > 0$, where the last line follows from 
the fact that $\pi(y)$ is $\bigo(n^{-2})$. 
It follows that that 
\begin{equation}
\l| p^t(x,y) - \pi(y) \r| \leq  {A \over t},  %\\
%&\leq& {2 \over n^2} + {B \over t}.
\end{equation}
for all $t \geq 1$, and the proof is complete.
\end{proof}

\section{Appendix B}

\begin{lemma}
\label{randomwalk}
Let $W_t = (X_t, Y_t)$ be a simple random walk on $\z^2$,
started at $(0,1)$.
Fix a positive integer $k$ and let
$A, B$ and $C$ be the lines $y = 0$, $y = k$ and 
$|x| = k$, respectively. Let $\tb$ and $\tc$ be the hitting 
times of $A \cup B$ and $A \cup C$, respectively.
\begin{description}
\item{{\bf (i)}} $\,$ 
\[
\P( W_{T_B} \in B)  = {1 \over n}.
\]
\item{{\bf (ii)}}  $\,$
\[
\P( W_{T_C} \in C) \leq {2 \over n}.
\]
\end{description}
\end{lemma}
\begin{proof}
\noindent{{\bf (i)}} $\;$ This is immediate by the optional stopping theorem 
because $Y_t$ is a bounded martingale and $T_B$ is 
a stopping time.

\noindent{{\bf (ii)}} $\;$  
Let $T = \min( \tb, \tc)$. 
Note that $T_C$ and $T$ 
are stopping times. 
A routine calculation shows that $Y_t^2 - X_t^2$ is a 
martingale. It follows that 
$Y_{\tmt}^2 - X_{\tmt}^2$ is a bounded submartingale. Thus the optional 
stopping theorem implies that
\[
\e \left( Y_T^2 - X_T^2 \right) = \e \left( Y_0^2 - X_0^2 \right) = 1,
\]
and hence
\begin{equation}
\label{star}
\e(X_T^2) < \e(Y_T^2).
\end{equation}
But since $Y_{\tmt}^2$ is a bounded submartingale and $T \leq T_B$,
we have 
\begin{eqnarray*}
\e \left( Y_T^2 \right) &\leq& \e \left( Y_{\tb}^2 \right) \\
&=& k^2 \P( Y_{\tb} = k ) \\
&=& k,
\end{eqnarray*}
where the last line holds because 
$\P( Y_{\tb} = k ) = {1 \over k}$ by part (i) of the lemma. 
Combining this with (\ref{star}) gives
\begin{equation}
\label{eboundy}
\e \left( X_T^2 + Y_T^2 \right) < 2k.
\end{equation}
It follows that
\begin{eqnarray*}
\P( W_T \in B \cup C) &=& \P( X_T^2 + Y_T^2 \geq k^2 ) \\
&\leq& {1 \over k^2} \e( X_T^2 + Y_T^2 ) \\
&\leq& {2 \over k },
\end{eqnarray*}
where first inequality is Markov's and the second follows 
from (\ref{eboundy}).
This verifies (ii) because $W_T \in B \cup C$ whenever $W_T \in C$. 
\end{proof}

{\bf Acknowledgments.} We are  grateful to 
I.~Benjamini for bringing the problem of the mixing time of the 
fifteen puzzle to our attention.

\end{document}